\documentclass[final,hidelinks,onefignum,onetabnum]{siamart220329}

\usepackage{amssymb,amsfonts,setspace,color,mathtools,graphicx,booktabs,multirow,bm,algorithmic}
\graphicspath{{Graphics/}}

\newtheorem{defn}{Definition}

\DeclarePairedDelimiter\floor{\lfloor}{\rfloor}

\newcommand{\bv}{{\boldsymbol{v}}} 
\newcommand{\bx}{{\boldsymbol{x}}} 
\newcommand{\by}{{\boldsymbol{y}}} 
\newcommand{\bz}{{\boldsymbol{z}}} 

\newcommand{\bA}{{\boldsymbol{A}}} 
\newcommand{\bC}{{\boldsymbol{C}}} 
\newcommand{\bD}{{\boldsymbol{D}}} 
\newcommand{\bX}{{\boldsymbol{X}}} 
\newcommand{\bY}{{\boldsymbol{Y}}} 

\newcommand{\xj}{\boldsymbol{x}_i}

\newcommand{\yrj}{\boldsymbol{y}_i^{+r}}
\newcommand{\Aj}{\boldsymbol{A}_i}
\newcommand{\Arrj}{\boldsymbol{A}_i^{+rr}}
\newcommand{\xrrj}{\boldsymbol{x}_i^{+rr}}
\newcommand{\xrj}{\boldsymbol{x}_i^{+r}}
\newcommand{\Uj}{\boldsymbol{U}_i}
\newcommand{\Urj}{\boldsymbol{U}_i^{+r}}

\newcommand{\Prrj}{\boldsymbol{P}_i^{+rr}}

\newcommand{\qrrj}{q_i^{+rr}}
\newcommand{\qrj}{q_i^{+r}}
\newcommand{\qj}{q_i}

\newcommand{\xonej}{\boldsymbol{x}_i^{+1}}
\newcommand{\qonej}{q_i^{+1}}
\newcommand{\Wj}{\boldsymbol{W}_i}
\newcommand{\Donej}{{\boldsymbol{D}_i^{(l)}}^{+1}}
\newcommand{\Dj}{\boldsymbol{D}_i^{(l)}}
\newcommand{\bj}{\boldsymbol{b}_{(l)}^{k,i-1}}
\newcommand{\Ponej}{\boldsymbol{P}_i^{+1}}
\newcommand{\Dvj}{\boldsymbol{D}_i^{(v)}}
\newcommand{\Dhj}{\boldsymbol{D}_i^{(h)}}
\newcommand{\bvj}{\boldsymbol{b}_{(v)}^{k,i-1}}
\newcommand{\bhj}{\boldsymbol{b}_{(h)}^{k,i-1}}

\newcommand{\Lip}{{\rm Lip}}
\newcommand{\TV}{{\rm TV}}

\newcommand{\matT}{^{\rm T}}

\newcommand{\gto}{\rightarrow}
\newcommand{\Brac}[1]{\left(#1\right)}
\newcommand{\Rectbrac}[1]{\left[#1\right]}
\newcommand{\norm}[1]{\left\|#1\right\|}
\newcommand{\normo}[1]{\|#1\|}
\newcommand{\norme}[1]{\left\|#1\right\|}

\newcommand{\mE}{\mathbb{E}}
\newcommand{\mR}{\mathbb{R}}

\newcommand{\mcJ}{\mathcal{J}}
\newcommand{\mcL}{\mathcal{L}}

\newcommand{\mdd}{\mathrm{d}}

\newsiamthm{asm}{Assumption}
\newsiamthm{rem}{Remark}

\title{
Local MALA-within-Gibbs for Bayesian image deblurring with total variation prior\thanks{Submitted to the editors on September 18, 2024.}
}

\author{Rafael Flock\thanks{Department of Applied Mathematics and Computer Science, Technical University of Denmark (\email{raff@dtu.dk}, \email{yido@dtu.dk}).}
\and Shuigen Liu\thanks{Department of Mathematics, National University of Singapore (\email{shuigen@u.nus.edu}, \email{mattxin@nus.edu.sg}).}
\and Yiqiu Dong\footnotemark[2]
\and Xin T. Tong\footnotemark[3]}

\headers{
Local MALA-within-Gibbs for Bayesian image deblurring
}
{Rafael Flock, Shuigen Liu, Yiqiu Dong and Xin T. Tong}

\ifpdf
\hypersetup{
pdftitle={
	Local MALA-within-Gibbs for Bayesian image deblurring with total variation prior
},
pdfauthor={Rafael Flock, Shuigen Liu, Yiqiu Dong and Xin T. Tong}
}
\fi

\begin{document}
\maketitle
\begin{abstract}
	We consider Bayesian inference for image deblurring with total variation (TV) prior.
	Since the posterior is analytically intractable, we resort to Markov chain Monte Carlo (MCMC) methods.
	However, since most MCMC methods significantly deteriorate in high dimensions, they are not suitable to handle high resolution imaging problems.
	In this paper, we show how low-dimensional sampling can still be facilitated by exploiting the sparse conditional structure of the posterior.
	To this end, we make use of the local structures of the blurring operator and the TV prior by partitioning the image into rectangular blocks and employing a blocked Gibbs sampler with proposals stemming from the Metropolis-adjusted Langevin Algorithm (MALA).
	We prove that this MALA-within-Gibbs (MLwG) sampling algorithm has dimension-independent block acceptance rates and dimension-independent convergence rate.
	In order to apply the MALA proposals, we approximate the TV by a smoothed version and show that the introduced approximation error is evenly distributed and dimension-independent.
	Since the posterior is a Gibbs density, we can use the Hammersley-Clifford Theorem to identify the posterior conditionals which only depend locally on the neighboring blocks.
	We outline computational strategies to evaluate the conditionals, which are the target densities in the Gibbs updates, locally and in parallel.
	In two numerical experiments, we validate the dimension-independent properties of the MLwG algorithm and demonstrate its superior performance over MALA.
\end{abstract} 



\section{Introduction}
Image acquisition systems usually capture only a blurred version of the ``true image''.
In many applications, such as remote sensing, medical imaging, astronomy, and digital photography, the aim is to reconstruct the “true image” from the acquired image \cite{berteroIntroductionInverseProblems2021a, chan2005image}.
Since the blurring can oftentimes be modeled by a linear convolution operator, image deblurring is a classical linear inverse problem.
In this paper, we follow the typical assumptions that the blurred image is obtained by a linear operator and that it is corrupted by additive Gaussian noise.
Computing a solution is usually not straightforward due to the ill-posedness of the problem, which is due to the ill-conditioning of the forward operator and the noise.
To resolve the ill-posedndess, we often require regularization, and a commonly used regularization technique in image reconstruction is the edge-preserving total variation regularization (TV) introduced in \cite{rudinNonlinearTotalVariation1992}. 

In this paper, we are not only interested in computing a reconstruction of the “true image”, but also in quantifying the uncertainty of the reconstruction.
To this end, we formulate the image deblurring problem with TV regularization as a Bayesian inverse problem.
Bayesian inverse problems can be characterized by the posterior density, which is the product of a likelihood function and a prior density \cite{gelmanBayesianDataAnalysis1995}.
Then, in accordance with the deterministic inverse problem described in the first paragraph, the likelihood function is linear-Gaussian and the prior density is a Gibbs density with a potential function given by TV.
This so-called TV prior was introduced for Bayesian inference in electrical impedance tomography in \cite{somersaloImpedanceImagingMarkov1997, kolehmainenBayesianApproachTotal1998} and has also been used in other Bayesian image reconstruction problems, e.g., image deblurring \cite{babacanVariationalBayesianBlind2009} and geologic structure identification \cite{leeBayesianInversionTotal2013a}. In \cite{lassasCanOneUse2004a}, it is shown that the TV prior is not edge-preserving when refining the discretization of the unknown image while keeping the discretization of the observed image fixed. 
We note that our work is not affected by this result, since we assume that the unknown and the observed image have the same discretization.

The posterior obtained from the linear-Gaussian likelihood and the TV prior is log-concave and can be expressed as a Gibbs density of the form
\begin{equation}\label{eq:pos_intro}
	\pi(\bx) \propto \exp(-l(\bx)-\varphi_0(\bx)),
\end{equation}
where the potential $l$ is continuously differentiable and gradient-Lipschitz, and the potential $\varphi_0$ is non-smooth. 
Target densities of this structure are common in Bayesian inverse problems with sparsity promoting priors \cite{agrawal2022variational}, which are often used in image reconstruction problems and sparse Bayesian regression, e.g., in the famous Bayesian LASSO \cite{parkBayesianLasso2008d}.

One way to perform uncertainty quantification with respect to a posterior is by sampling from it.
The samples can then be used to, e.g., approximate expectations via Monte Carlo estimates or to compute posterior statistics such as mean, standard deviation, or credibility intervals (CI).
Since we cannot sample directly from our posterior density in closed form, we use Markov chain Monte Carlo (MCMC) methods.
However, the high dimensionality of images prohibits the direct application of most MCMC methods, as their convergence slows down considerably with increasing dimension.

In this paper, we show how low-dimensional sampling can nevertheless be facilitated by exploiting the sparse conditional structure of the posterior.
The sparse conditional structure is due to the fact that the full conditional posterior of any image patch of arbitrary size only depends on the neighboring pixels within some radius.
This can be intuitively understood by considering that the convolution as well as TV operate \emph{locally} on the image.
Mathematically, the independence relationships among image patches can be proven by the Hammersley-Clifford Theorem because the considered posterior is a Gibbs density with a sparse neighborhood \cite{liMarkovRandomField2009}.
The Gibbs sampler is an attractive choice for posteriors with sparse conditional structure because the reduced dependencies only require a local evaluation of the conditionals, which is usually cheaper to compute.
Moreover, several updates may be performed in parallel.
To make use of these advantages, we partition the image into square blocks of equal size and employ a blocked Gibbs sampler with local and parallel updates of the image blocks.

Since we can not sample the conditionals in the blocked Gibbs sampler in closed form, we employ MALA-within-Gibbs (MLwG) sampling. 
In such a sampling scheme, one generates candidate samples via the well-known Metropolis-adjusted Langevin Algorithm (MALA) proposal.
MALA belongs to the class of Langevin Monte Carlo (LMC) methods, which are derived by discretizating a Langevin diffusion equation \cite{robertMonteCarloStatistical2004c}.
However, LMC methods require the gradient of the log-target density, and due to the non-smoothness of $\varphi_0$ in \cref{eq:pos_intro} we can not directly use the MALA proposal in the blocked Gibbs sampler.

Several adaptations of LMC methods to non-smooth target densities have been developed in the literature.
Arguably the most prominent adaptations are the proximal LMC algorithms, which were introduced in \cite{pereyraProximalMarkovChain2016a, durmusEfficientBayesianComputation2018a}.
Therein, the target is approximated by its Moreau-Yoshida envelope, which is continuously differentiable. 
An overview of proximal LMC methods is given in \cite{lauNonLogConcaveNonsmoothSampling2023}.
Another adapted LMC method is perturbed Langevin Monte Carlo (P-LMC), which is based on Gaussian smoothing \cite{nesterovRandomGradientFreeMinimization2017, pmlr-v108-chatterji20a}.
While both, proximal LMC and P-LMC, solve the issue of non-smoothness, their performance still deteriorates with increasing dimension.

In this work, instead of adapting the sampling method, we approximate the target density \cref{eq:pos_intro} by replacing $\varphi_0$ with a smooth approximation $\varphi_\varepsilon$:
\begin{equation}\label{eq:pos_smooth_intro}
	\pi_\varepsilon(\bx) \propto \exp(-l(\bx)-\varphi_\varepsilon(\bx)).
\end{equation}
The smoothed potential is such that $\varphi_\varepsilon \to \varphi_0$ as $\varepsilon \to 0$ and allows us to use the MALA proposal in the blocked Gibbs sampler.
For the resulting posterior approximation \cref{eq:pos_smooth_intro}, we derive a dimension-independent bound on the Wasserstein-$1$ distance between the marginal densities of $\pi$ and $\pi_\varepsilon$, which guarantees the accuracy of the smoothing.

It is shown in \cite{MR4111677} that under the assumption of sparse conditional structure, MLwG can have dimension-independent block acceptance rate (at fixed step size) and dimension-independent convergence rate. 
In this paper, we present similar guarantees with specific conditions tailored to our target density and taking the approximation error into account.
Finally, by making use of the sparse conditional structure of our posterior, which still holds under the smoothing, we show how an efficient \emph{local \& parallel} MLwG algorithm can be implemented.
We test the algorithm in two numerical experiments, where we illustrate the dimension-independent block acceptance rate and the dimension-independent convergence.
Moreover, we compare the MLwG algorithm to MALA and show that MLwG clearly outperforms MALA in terms of sample quality and computational wall-clock time with increasing dimension.

\vspace{12pt}

\paragraph{\texorpdfstring{\bf Contributions}{Contributions}}
We now summarize our main contributions.
\begin{itemize}
	\item We develop a Stein's method based approach to prove that the error between the low dimensional marginals of the target distribution $\pi$ in \cref{eq:pos_intro} and the smoothed distribution $\pi_\varepsilon$ in \cref{eq:pos_smooth_intro} is dimension-independent.  	
	\item We show how to implement an efficient local \& parallel MLwG sampling algorithm by providing the local target densities and their gradients for the block updates.
	\item We illustrate the dimension-independent block acceptance and convergence rates in two numerical experiments, where we also show that MLwG clearly outperforms MALA.
\end{itemize}

\vspace{12pt}

\paragraph{\texorpdfstring{\bf Notation}{Notation}}
In this paper, we work with discrete and square images of $n\times n$ pixels, but the results can be extended to rectangular images.
We employ a vector notation for the images by stacking them in the usual column-wise fashion.
Furthermore, we partition the images into blocks of size $m\times m$, such that $n/m$ is an integer.
We denote the number of pixels in one image by $d=n^2$, the number of pixels in one block by $q=m^2$, and the number of blocks by $b=(n/m)^2$.
We denote pixels by lower case Greek letters and blocks by lower case Roman letters.
E.g., for an image $\bx\in\mathbb{R}^d$, we write $\bx_\alpha\in\mathbb{R}$, and  $\bx_i\in\mathbb{R}^q$.
We write $[b] \coloneqq \{1,2,\dots,b\}$ where $b$ is some positive integer, and $\setminus i \coloneqq [b]\setminus i$ for $i\in[b]$.
We use upper case Greek letters to denote sets of pixel indices.
For example, $\bx_\Theta$ is the parameter block of pixels with indices in $\Theta\subseteq[d]$.
We use the notation $\nabla_i := \nabla_{\bx_i}$ for the gradient operator with respect to the pixels of block $\bx_i$.

\vspace{12pt}

\paragraph{\texorpdfstring{\bf Outline}{Outline}}

The remainder of this paper is organized as follows.
In \Cref{sec:prelim}, we formulate the Bayesian deblurring problem and recall the blocked MLwG sampler.
In \Cref{sec:pos_smooth}, we propose the smoothing of the posterior and bound the error in the marginals in the Wasserstein-$1$ distance.
In \Cref{sec:dim_ind_acc_conv}, we present the dimension-independent block acceptance rates and dimension-independent convergence rate of MLwG when applied to the smoothed posterior based on results from \cite{MR4111677}.
In \Cref{sec:loc_par_samp}, we present the local \& parallel MLwG algorithm and show how the target densities, i.e., the conditionals, can be evaluated locally and in parallel.
In \Cref{sec:num_ex}, we validate the dimension-independent properties of the local \& parallel MLwG algorithm in two numerical examples and perform a comparison to MALA. 
We conclude this paper in \Cref{sec:conc}.

\section{Preliminaries}\label{sec:prelim}
\subsection{Problem setting}
We fist consider the classic image deconvolution problem with TV regularization, e.g., \cite{Paragios_Chen_Faugeras_2006} and assume that a blurred and noisy image $\by\in\mathbb{R}^d$ is obtained by
\begin{equation}\label{eq:obs_model}
	\by = \bA \bx_\mathrm{true} + \boldsymbol{\epsilon}.
\end{equation}
Here, $\bx_\mathrm{true}\in\mathbb{R}^d$ is the “true” image, $\boldsymbol{\epsilon}\sim\mathcal{N}(\boldsymbol{0}, \tfrac{1}{\lambda} \boldsymbol{I}_d)$, and $\bA\in\mathbb{R}^{d\times d}$ is the convolution operator.
In particular, one can construct $\bA$ via the discrete point spread function (PSF), i.e., the convolution kernel.
We assume that the discrete PSF has radius $r>0$, such that $\bA\bx$ convolves each pixel with the surrounding $(2r+1)^2$ pixels.

\Cref{eq:obs_model} constitutes an inverse problem where the goal is to recover a solution that is “close to $\bx_\mathrm{true}$” from the data $\by$.
Computing a solution is typically not straightforward due to the ill-posedness of the problem.
That is, $\bA$ may be badly conditioned and thus highly sensitive to the noise.
For this reason, we employ the edge-preserving TV regularization introduced in \cite{rudinNonlinearTotalVariation1992}, which is a commonly used regularization technique in image reconstruction.
For discretized images, it reads 
\begin{equation}\label{eq:TV}
	\TV(\bx) = \sum_{\alpha=1}^d \sqrt{ (\bD_\alpha^{(v)} \bx)^2 + (\bD_\alpha^{(h)} \bx)^2 },
\end{equation}
where 
$\bD^{(v)}\in\mathbb{R}^{d\times d}$ and $\bD^{(h)}\in\mathbb{R}^{d\times d}$ 
are finite difference matrices for the computation of the horizontal and vertical differences between the pixels.
The subscript $\alpha$ denotes the $\alpha$-th row of $\bD^{(v)}$ and $\bD^{(h)}$, such that
$\bD_\alpha^{(v)}\bx$ and $\bD_\alpha^{(h)}\bx$ are the differences between pixel $\bx_\alpha$ and its neighboring pixels in the vertical and horizontal directions, respectively.
The finite difference matrices are defined as
\begin{equation}\label{eq:D_def}
	\begin{aligned}
		\bD^{(v)} &= \boldsymbol{I}_n \otimes \bD_n \\
		\bD^{(h)} &= \bD_n \otimes \boldsymbol{I}_n,
	\end{aligned}
	\qquad \mathrm{with} \qquad
	\bD_n = \begin{bmatrix}
		-1 &1 & & &  \\
		&-1 &1 & &  \\
		& &\ddots &\ddots &  \\ 
		& & &-1 &1 \\
		& & & &-1 \end{bmatrix}_{[n\times n]}.
\end{equation}
Here, $\otimes$ denotes the Kronecker product.
In this paper, without loss of generality, we consider zero boundary conditions for the finite difference matrices. 
Other boundary conditions can be used by slight modification of the algorithm and the analysis.
In practice, we compute the convolution and TV in a matrix-free way due to the high dimensionality of the problems.
However, the matrix expressions are more convenient for the derivation of our theoretical results.

We now formulate the Bayesian inverse problem by defining the posterior probability density
$\pi(\bx|\by) \propto \pi(\by|\bx) \pi_{\rm prior}(\bx)$
with the likelihood function $\bx \mapsto \pi(\by|\bx)$ and the prior density $\pi_{\rm prior}(\bx)$.
The likelihood function is determined by the data generating model \cref{eq:obs_model} and reads
\begin{equation}\label{eq:like}
	\pi(\by|\bx) \propto \exp\left( -\frac{\lambda}{2} \| \by - \bA \bx \|_2^2 \right).
\end{equation}
We note that the data $\by$ is fixed in this paper.
The prior is constructed based on the discretized TV \cref{eq:TV} and is of the Gibbs-type density
\begin{equation} \label{eq:TV_prior}
	\pi_{\rm prior}(\bx) \propto \exp\left( - \delta \TV(\bx) \right),
\end{equation}
where $\delta>0$ is controlling the strength of the prior.
In this paper, we consider $\delta$ to be given.
The prior \cref{eq:TV_prior} is called the TV prior and was introduced in \cite{somersaloImpedanceImagingMarkov1997, kolehmainenBayesianApproachTotal1998}.
Consequently, the likelihood function \cref{eq:like} and the TV prior \cref{eq:TV_prior} give rise to a log-concave composite posterior density of the form
\begin{equation}\label{eq:pos}
	\begin{split}
		\pi(\bx) := \pi(\bx|\by)  &\propto \exp\left( - l(\bx) - \varphi_0(\bx) \right), \\
		\text{with the potentials} \quad l(\bx) &:= \frac{\lambda}{2} \| \by - \bA \bx \|_2^2 \quad \text{and} \quad \varphi_0(\bx) := \delta \TV(\bx).
	\end{split}
\end{equation}
\subsection{MALA-within-Gibbs (MLwG)}\label{sec:MLwG_not_alg}
A Gibbs sampler in its original form, also called component-wise or sequential Gibbs sampler, updates each component $\bx_\alpha$ sequentially from the full conditional 
$\pi(\bx_\alpha|\bx_{\setminus\alpha})$,
and a new sample $\bx$ is obtained after all components are updated. 
It can be shown that a sample chain constructed by such an algorithm converges to the target density, see, e.g., \cite{robertMonteCarloStatistical2004c}.
Moreover, the convergence result holds also for \emph{block updates}, i.e., \emph{blocked Gibbs sampling}, which we employ in this paper.

To this end, we partition the image $\bx$ into $b$ square block images $\bx_i \in\mathbb{R}^q$, $i=1,\dots,b$.
See \Cref{fig:upd_sched} for an example with a 4-by-4 block partition.
The blocks $\bx_i$ have equal side lengths $m\in\mathbb{N}$, such that $n/m$ is an integer, and they contain $q=m^2$ pixels. 
Moreover, we require $m>2r$, where $r$ is the radius of the discrete PSF in the convolution.
\begin{figure}[htbp]
	\centering
	\includegraphics[width=0.5\textwidth]{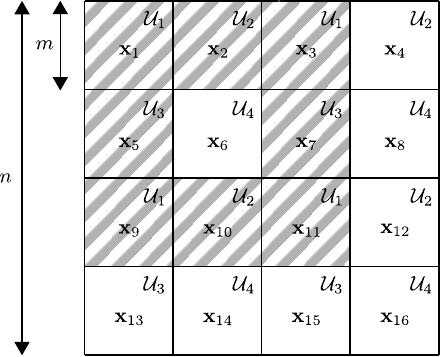}
	\caption{Example of a partition with 16 blocks. 
		When updating a block, the neighboring blocks must be fixed.
		This is exemplified here for block $6$, where the hatched blocks must be fixed.
		A possible parallel updated scheme for MLwG is indicated by the index sets $\mathcal{U}_l,\,l=1\dots4$, where for a given $l$, all blocks associated to $\mathcal{U}_l$ can be updated in parallel.}
	\label{fig:upd_sched}
\end{figure}

We now establish some notation for the blocked Gibbs sampler.
We call a complete iteration in which all blocks are updated a \textit{cycle} and denote by $\bx^{k,i}\in\mathbb{R}^d$ the state \emph{during} the $k$-th cycle \textit{after} the update of the $i$-th block. 
To illustrate this notation, consider the following presentation of block updates: 
\begin{equation}\label{eq:mlwg_states}
	\bx^0 \gto \underbrace{\bx^{1,1} \gto \bx^{1,2} \gto \cdots \gto \bx^{1,b}}_{\text{$1$st cycle}} =: \bx^1 \gto \underbrace{\bx^{2,1} \gto \bx^{2,2} \gto \cdots \gto \bx^{2,b}}_{\text{$2$nd cycle}} =: \bx^2 \gto \cdots     
\end{equation}
Notice also that we introduce $\bx^k$ to denote the state \emph{after} the $k$-th cycle.
Other block updating rules like randomized sequences could also be employed, but are not discussed here.

Now consider the state $\bx^{k,i-1}$ of the sample chain \cref{eq:mlwg_states}.
To move to state $\bx^{k,i}$, we fix $\bx^{k,i}_{\setminus i}=\bx^{k,i-1}_{\setminus i}$ and sample the remaining block
$\bx_i^{k,i} \sim \pi(\bx_i^{k,i}|\bx^{k,i-1}_{\setminus i})$.
If it is not possible to sample $\bx_i^{k,i} \sim \pi(\bx_i^{k,i}|\bx^{k,i-1}_{\setminus i})$ in closed form, one can use a single Metropolis-Hastings step as follows.
A candidate $\bz^{k,i}_i$ is simulated from a proposal density $q(\cdot|\bx^{k,i-1}_i)$ and accepted with probability
\begin{equation}\label{eq:ac_rate}
	\alpha^{k,i} = 1 \wedge \frac{ \pi(\bz_i^{k,i}|\bx^{k,i-1}_{\setminus i}) q(\bx^{k,i-1}_i|\bz^{k,i}_i) }
	{ \pi(\bx_i^{k,i-1}|\bx^{k,i-1}_{\setminus i}) q(\bz^{k,i}_i|\bx^{k,i-1}_i) }.
\end{equation}
Note that the proposal density $q(\cdot|\bx^{k,i-1}_i)$ depends on $\bx^{k,i-1}_{\setminus i}$, but here we omit the dependency for clearness. 
In MLwG, the MALA proposal density is used.
For the target density $\pi(\bx_i^{k,i}|\bx^{k,i-1}_{\setminus i})$, it reads
\begin{equation}\label{eq:MALA_prop_dens}
	q(\bz_i^{k,i}|\bx_i^{k,i}) = \exp \Brac{ - \frac{1}{4\tau} \| \bz_i^{k,i}  - \bx_i^{k,i} - \tau \nabla_i \log \pi(\bx_i^{k,i}|\bx^{k,i-1}_{\setminus i}) \|_2^2 },
\end{equation}
and a candidate is generated as
\begin{equation}\label{eq:MALA_proposal}
	\bz_i^{k,i} = \bx_i^{k,i} + \tau \nabla_i \log \pi(\bx_i^{k,i}|\bx^{k,i-1}_{\setminus i}) + \sqrt{2\tau} \boldsymbol{\xi}_i^{k,i}, \quad \boldsymbol{\xi}_i^{k,i} \sim\mathcal{N}(\boldsymbol{0},\boldsymbol{I}_q).
\end{equation}

Here, $\tau>0$ is the step size or scale parameter. 
The choice of the step size is crucial for the performance of MALA, as too large step sizes lead to many rejected samples, and too small step sizes lead to samples which are very close to each other.
In both cases, the correlation between the samples becomes large which is not preferable.
Several adaptive strategies for the tuning of the step size exist, and in our numerical experiments in \Cref{sec:num_ex}, we use the method from \cite{marshallAdaptiveApproachLangevin2012b} which adapts the step size during burn-in.
In the MLwG sampler, the adaptive strategies can be used to tune the step size for each block individually and automatically such that the algorithm requires little parameter tuning in general. 

We summarize the blocked MLwG sampler in \cref{alg:MLwG}.
\begin{algorithm}[htbp]
	\caption{MALA-within-Gibbs (MLwG) sampler}
	\label{alg:MLwG}
	\begin{algorithmic}[1]
		\REQUIRE Number of samples $N$, step size $\tau$, initial state $\bx^{0}$.
		\FOR{$k=1\dots N$}
		\STATE Set $\bx^{k,0}\gets\bx^{k-1}$.
		\FOR{$i=1\dots b$} \label{algline:it_k}
		\STATE Set $\bx^{k,i}\gets\bx^{k,i-1}$.
		\STATE Propose block candidate $\bz^{k,i}_i$ according to \cref{eq:MALA_proposal}.
		\STATE Compute acceptance probability $\alpha^{k,i}$ based on \cref{eq:ac_rate}.
		\STATE Draw $u\sim\mathrm{Uniform}(0,1)$.
		\IF{$\alpha^{k,i} > u$} 
		\STATE Set $\bx^{k,i}_i \gets \bz_i^{k,i}$.
		\ENDIF
		\ENDFOR
		\ENDFOR
	\end{algorithmic}
\end{algorithm}

\section{Posterior smoothing with dimension-independent error}\label{sec:pos_smooth}
Since MLwG requires the gradient of the log posterior density in \cref{eq:MALA_prop_dens,eq:MALA_proposal}, we propose to approximate the non-smooth posterior density $\pi$ in \cref{eq:pos} by a smooth density $\pi_\varepsilon$.
Moreover, we show that the introduced error between $\pi$ and $\pi_\varepsilon$ is distributed uniformly, leading to a local dimension-independent error on the marginal distribution of any block $\bx_i$. 

The non-smoothness of the posterior \cref{eq:pos} originates from the potential $\varphi_0(\bx)= \TV(\bx)$ in \cref{eq:TV}.
Hence, we replace $\varphi_0$ with a smoothed potential $\varphi_\varepsilon$ for some small $\varepsilon>0$, such that $\varphi_\varepsilon \to \varphi_0$ as $\varepsilon \to 0$. 
Various smoothing methods exist, but it is crucial to ensure that the introduced error remains small, particularly when $\bx$ resides in a high-dimensional space. In this paper, we consider the approximation 
\begin{equation}\label{eq:approxTV}
	\varphi_\varepsilon(\bx) := \delta \sum_{\alpha=1}^d \sqrt{ (\bD_\alpha^{(v)} \bx)^2 + (\bD_\alpha^{(h)} \bx)^2 + \varepsilon }, 
\end{equation}
for some small $\varepsilon>0$, see, e.g., \cite{vogelComputationalMethodsInverse2002}.
Thus, our smoothed posterior density reads
\begin{equation}\label{eq:smooth_pos_expl}
	\pi_\varepsilon(\bx) \propto \exp\left( - \frac{\lambda}{2} \| \by - \bA \bx \|_2^2 - \delta \sum_{\alpha=1}^d \sqrt{ (\bD_\alpha^{(v)} \bx)^2 + (\bD_\alpha^{(h)} \bx)^2 + \varepsilon } \right).
\end{equation}

As we modify $\varphi_0$, a function in $\mathbb{R}^d$, the distance between $\pi$ and $\pi_\varepsilon$ generally depends on the dimension $d$. 
For instance, one can show $\text{KL}(\pi|\pi_\varepsilon) = \mathcal{O}(d\varepsilon)$.
Here, $\text{KL}(\pi|\pi_\varepsilon)$ is the Kullback-Leibler divergence of $\pi$ from $\pi_\varepsilon$, see, e.g., \cite{Pardo_2018}.
However, when examining the marginals of $\pi$ and $\pi_\varepsilon$ over small blocks $\bx_i$, we can show that the approximation error is independent of the full dimension and thus dimension-independent. 
We comment that such dimension independence is crucial for solving the image deblurring problem. 
It ensures that the smoothing error is evenly distributed across the image, rather than concentrating on certain pixels and creating unwanted artifacts in the image. 
Ensuring a uniformly distributed error is not a concept unique to image reconstruction, for other applications see \cite{fan2018ell,tong2023pca,hu2024network}.

The key to achieve the desired result lies in the sparse conditional structure. Intuitively, such local structure makes the distribution on small blocks mostly influenced by local modifications to $\varphi_0$, and modifications on remote blocks have little impact on it. To rigorously justify the above observations, we introduce the concept of $c$-diagonal block dominance, which is imposed on $\bC = \bA\matT \bA$ to quantify how locally concentrated the blurring matrix $\bA$ is.

\begin{defn}    \label{def:local}
	Consider a positive definite matrix $\bC \in\mR^{d\times d}$ and let $\bC_{ij}$ denote the $(i,j)$-th sub block of $\bC$ of size $q$. 
	$\bC$ is called \textbf{$c$-diagonal block dominant} for some $c>0$, if there exists a symmetric matrix $M \in \mR^{b\times b}$ s.t.\
	\begin{enumerate}
		\setstretch{1.25}
		\item for any $i \in [b]$, $\bC_{ii} \succeq M_{ii} I$,
		\item for any $i,j \in [b]$, $i\neq j$, $\norm{\bC_{ij}}_2 \leq M_{ij}$,
		\item for any $i \in [b]$, $\sum_{j\neq i} M_{ij} + c \leq M_{ii} $. 
	\end{enumerate}
\end{defn}
\begin{rem}
\cite{MR4111677} introduces a similar blockwise log-concavity condition, which is crucial for the dimension-independent convergence rate therein. The $c$-diagonal block dominance here can be viewed as an $\ell_1$ version of the blockwise log-concavity condition (which is an $\ell_2$ condition).
\end{rem}

Now we state the main theorem, which shows that the marginal error of $\pi$ and $\pi_\varepsilon$ is dimension-independent if $\bA\matT\bA$ is $c$-diagonal block dominant.
\begin{theorem} \label{thm:TV_ApproxErr}
	Consider the target distribution $\pi$ defined in \cref{eq:pos} and its smooth approximation $\pi_\varepsilon$ in \cref{eq:smooth_pos_expl}. Assume that $\bA\matT \bA$ is $c$-diagonal block dominant as in \cref{def:local}. 
	Suppose $\frac{\lambda}{\delta} \geq \frac{64m}{c \sqrt{\varepsilon}}$. 
	Then, there exists a dimension-independent constant $C$ such that 
	\begin{equation}
		\max_i W_1(\pi_i,\pi_{\varepsilon,i}) \leq C \varepsilon.
	\end{equation}
	Here, $\pi_i$ and $\pi_{\varepsilon,i}$ denote the marginal distributions of $\pi$ and $\pi_\varepsilon$ on $\bx_i$, respectively, and $W_1$ denotes the Wasserstein-$1$ distance
	\[
	W_1(\mu,\nu) := \inf_{ \gamma \in \Pi(\mu,\nu) } \int \norme{\bx-\by}_2 \gamma(\bx,\by) \mdd \bx \mdd \by, 
	\]
	where $\Pi(\mu,\nu)$ denotes all the couplings of $\mu$ and $\nu$, i.e., if $(X,Y)\sim \gamma \in \Pi(\mu,\nu)$, then $ X \sim \mu $ and $ Y \sim \nu$.
\end{theorem}
\begin{proof}
	See \Cref{app:ApproxErr}.
\end{proof}

\section{Dimension-independent block acceptance and convergence rate}\label{sec:dim_ind_acc_conv}
It is shown in \cite{MR4111677} that MLwG given in \cref{alg:MLwG} has dimension-independent block acceptance rates and a dimension-independent convergence rate for smooth densities under assumptions on sparse conditional structure and blockwise log-concavity. 
In the following, we present two results, which show that we can obtain similar dimension-independent results for $\pi_\varepsilon$ by using the concept of $c$-diagonal block dominance in \Cref{def:local}. 
Since the proofs are similar to those for Proposition 3.3 and Theorem 3.6 in \cite{MR4111677}, we skip them here.
\begin{proposition}\label{prop:ac_rate_bound}
	Suppose $\bA$ is bounded in the sense that
	\begin{equation}    \label{eq:asm_A_bound}
		\exists C_\bA>0 \quad \text{s.t.} ~\forall i,j \in [b], ~ \norm{ \bA_{ij} }_2 \leq C_\bA.
	\end{equation}
	Then the expected acceptance rate $\mE \alpha^{k,i}$ \cref{eq:ac_rate} of the MLwG proposal \cref{eq:MALA_proposal} for the smoothed distribution $\pi_\varepsilon$ is bounded below by 
	\begin{equation}\label{eq:ac_rate_lound}
		\mE \alpha^{k,i} \geq 1 - M \sqrt{\tau}. 
	\end{equation}
	Here, $M$ is a dimension-independent constant depending on the block size $q$, $C_\bA, \delta, \varepsilon$, and $\max_j \norme{ [\bA \bx^{k,i-1} - \by]_j }_2$.
\end{proposition}

\begin{rem}
In fact, the lower bound depends on the state of x through the term $\max_j \norme{ [\bA \bx^{k,i-1} - \by]_j }_2$.
However, since the pixels values are bounded in practice, this term is usually also bounded and dimension-independent.
\end{rem}

\begin{proposition} \label{prop:conv}
	Suppose $\bA$ is bounded as in \cref{eq:asm_A_bound}, $\bA\matT \bA$ is $c$-diagonally block dominant, and $\frac{\lambda}{\delta} \geq \frac{64m}{c \sqrt{\varepsilon}}$. 
	Then, there exists $\tau_0 > 0$, independent of the number of blocks $b$, such that for all $0 < \tau \leq \tau_0$, we can couple two MLwG samples $\bx^k$ and $\bz^k$ such that 
	\[
	\sum_{j=1}^b \Rectbrac{\mE \norm{\bx_j^k - \bz_j^k}_2}^2 \leq \Brac{ 1 - \lambda c \tau /4 }^{2k} \sum_{j=1}^b \Rectbrac{\mE \norm{\bx_j^0 - \bz_j^0}_2}^2. 
	\]
	Here, $\tau_0$ ensures that $\lambda c \tau / 4 \in (0,1)$. 
	In particular, one can let $\bz^0 \sim \pi_\varepsilon$, and it follows that $\bz^k \sim \pi_\varepsilon$, which in turn shows that $\bx^k$ converges to $\pi_\varepsilon$ geometrically fast. 
\end{proposition}

\begin{rem}
The convergence result is based on the maximal coupling of two MLwG chains as in \cite{MR4111677}. In brief, the two chains are coupled to share in each step the same $\xi_i^{k,i}$ in \eqref{eq:MALA_proposal} and the random variable $u\sim\text{Uniform}(0,1)$ used to determine the acceptance of proposals (See \Cref{alg:MLwG}). 
\end{rem}

\section{\texorpdfstring{Local \& parallel MLwG algorithm}{Local and parallel MLwG algorithm}}\label{sec:loc_par_samp} 

The posterior \cref{eq:smooth_pos_expl} can be interpreted as a Gibbs density over an undirected graph.
The nodes are given by the pixels, and the edges are given by the neighborhoods of the pixels.
In particular, a pixel is connected to all its neighbors.
Pixels are neighbors to each other if their associated random variables directly interact with each other in the joint density.
For example, two random variables $Z_1\in\mathbb{R}$ and $Z_2\in\mathbb{R}$ directly interact in the density $\rho(Z)\propto\exp(-Z_1 Z_2)$ but not in $\rho(Z)\propto\exp(- (Z_1 + Z_2) )$.
A set of pixels in which all pixels are neighbors to each other is called clique.
Furthermore, a maximal clique is a clique to which no pixel can be added such that the set is still a clique \cite{liMarkovRandomField2009}.

In the posterior \cref{eq:smooth_pos_expl}, the maximal cliques can be associated to the pixels.
This is due to the convolution, where all pixels within the convolution radius have a direct relationship with each other.
A Gibbs density can be factorized over its maximal cliques, and thus we can write for the posterior \cref{eq:smooth_pos_expl}
\begin{equation}\label{eq:clique_pots}
	\pi(\bx) \propto \exp{\Big( - \sum_{\alpha\in[d]} V_\alpha( \bx ) \Big)},
\end{equation}
where the maximal clique potentials $V_\alpha$ are given by
\begin{equation*}
	V_\alpha( \bx ) = \frac\lambda2 ( \by_\alpha -  \bA_\alpha \bx )^2 + \delta \sqrt{ (\bD_\alpha^{(v)} \bx)^2 + (\bD_\alpha^{(h)} \bx)^2 + \varepsilon }.
\end{equation*}
In fact, each clique potential $V_\alpha$ only depends on the pixels belonging to the clique, but for now we write $V_\alpha(\bx)$.

\begin{lemma}\label{lem:cond_one_pixel}
	Let $\Theta_\alpha\subseteq[d]$ be the index set containing $\alpha$ and the indices of all pixels within a frame of radius $r$ around pixel $\bx_\alpha$.
	Analogously, let $\Phi_\alpha\subseteq[d]$ be the index set containing $\alpha$ and the indices of all pixels within a frame of radius $2r$ around pixel $\bx_\alpha$.
	Here, the integer $r>0$ is the radius of the discrete PSF.
	Then, we can express the full conditional of $\bx_\alpha$ as 
	\begin{equation}\label{eq:cond_pixel}
		\bx_\alpha \sim \pi(\bx_\alpha|\bx_{\setminus \alpha}) = \pi(\bx_\alpha|\bx_{\Phi_\alpha \setminus \alpha}) \propto \exp{ \Big( - \sum_{\beta\in\Theta_\alpha} V_\beta( \bx_{\Phi_\alpha} ) \Big)}.
	\end{equation}
\end{lemma}
\begin{proof}
	See \Cref{app:loc_cond}. 
\end{proof}

In \Cref{lem:cond_one_pixel}, $\Theta_\alpha$ is the index set of all clique potentials that depend on $\bx_\alpha$, and $\Phi_\alpha$ is the index set of all pixels on which the clique potentials $\{V_\alpha\,|\,\alpha\in\Theta_\alpha\}$ depend.
\Cref{lem:cond_one_pixel} relies on the well-known Hammersly-Clifford theorem \cite{hammersley1971markov}, and \cref{eq:cond_pixel} basically states that pixel $\bx_\alpha$ and all pixels in $[d]\setminus\Phi_\alpha$ are conditionally independent given the pixels in $\Phi_\alpha\setminus\alpha$.
Or in other words, the full conditional of pixel $\bx_\alpha$ depends only on its neighborhood given by $\Phi_\alpha\setminus\alpha$.
These reduced conditional dependence relationships are generally known as \emph{sparse conditional structure}, and we equivalently say that the posterior density \cref{eq:smooth_pos_expl} is \emph{local} \cite{MR3901708,MR4111677}.

Sparse conditional structure is often a prerequisite for an efficient and dimension-independent MLwG sampler.
Indeed, \Cref{lem:cond_one_pixel} allows us to equivalently express the target $\pi(\bx_i|\bx_{\setminus i})$ of the blocked MLwG \Cref{alg:MLwG} by a density with reduced conditional dependencies.
Concretely, we can write similarly to \cref{eq:cond_pixel}
\begin{equation}\label{eq:red_cond_blocks}
	\bx_i \sim \pi(\bx_i|\bx_{\setminus i}) = \pi(\bx_i|\bx_{\Phi_i \setminus i}) \propto \exp{ \Big( - \sum_{\beta\in\Theta_i} V_\beta( \bx_{\Phi_i} ) \Big)},
\end{equation}
where $\Theta_i$ is now the index set of all clique potentials that depend on the pixels in block $i$, and $\Phi_i$ is the index set of all pixels on which the clique potentials $\{V_\alpha\,|\,\alpha\in\Theta_i\}$ depend.

In the following sections, we outline how $\exp{ ( - \sum_{\beta\in\Theta_i} V_\beta( \bx_{\Phi_i} ) )}$ and its gradient, which is required for the MALA proposal \cref{eq:MALA_proposal}, can be evaluated efficiently.
To this end, for the $(k,i-1)$-th state of the sample chain, we define 
\begin{itemize}
\setstretch{1.5}
    \item the \emph{local block likelihood}: $\exp\left( -l^{k,i-1}(\bx_i^{k,i}) \right)$ and
    \item the \emph{local block prior}: $\exp\left( -\varphi_\varepsilon^{k,i-1}(\bx_i^{k,i}) \right)$
\end{itemize} 
with the potentials $l^{k,i-1}:\mathbb{R}^q\to\mathbb{R}_{\geq0}$ and $\varphi_\varepsilon^{k,i-1}:\mathbb{R}^q\to\mathbb{R}_{\geq0}$, such that
\begin{equation}\label{eq:loc_block_dens}
	\bx_i^{k,i} \sim \exp{ \Big( - \sum_{\beta\in\Theta_i} V_\beta( \bx_i^{k,i}, \bx_{\Phi_i \setminus i}^{k,i} ) \Big)} = 
	\exp\left( -l^{k,i-1} (\bx_i^{k,i})  -\varphi_\varepsilon^{k,i-1}(\bx_i^{k,i}) \right).
\end{equation}
In the potentials $l^{k,i-1}$ and $\varphi_\varepsilon^{k,i-1}$, the pixels $\bx_{\Phi_i \setminus i}$ are fixed at the $(k,i-1)$-th state of the sample chain, i.e., $\bx_{\Phi_i \setminus i}^{k,i}=\bx_{\Phi_i \setminus i}^{k,i-1}$, such that $l^{k,i-1}$ and $\varphi_\varepsilon^{k,i-1}$ are functions in $\bx_i^{k,i}$.
In the following sections, we omit the superscripts $k,i$ of $\bx_i^{k,i}$ for readability and outline separately the explicit computations of $l^{k,i-1}(\bx_i)$ and $\varphi_\varepsilon^{k,i-1}(\bx_i)$, and their gradients.
\subsection{Local block likelihood}
To evaluate the convolution in $l^{k,i-1}(\bx_i)$, we require all pixels in the set $\Phi_i$, which contains the pixels in block $i$ and all pixels within a frame of radius $2r$ around block $i$.
In the following, we denote this extended block by $\xrrj \eqqcolon \bx_{\Phi_i} \in\mathbb{R}^{\qrrj}$.
Correspondingly, we let $\xrj\in\mathbb{R}^{\qrj}$ denote the block $\xj$ extended by all pixels within a frame of radius $r$ around block $i$.
Note that the extended blocks $\xrrj$ and $\xrj$ have different sizes $\qrrj$ and $\qrj$, respectively, based on their location in the image.
For example, $\xrrj$ has size $\qrrj\in\{(m+2r)^2,(m+4r)(m+2r),(m+4r)^2\}$, depending on whether block $i$ is located in a corner, at an edge or in the interior of the image.
Moreover, we define the component selection matrices $\Urj\in\mathbb{R}^{\qrrj\times\qrj}$ and $\Uj\in\mathbb{R}^{\qrrj\times\qj}$, which allow for the following mappings among $\xrrj$, $\xrj$, and $\xj$:
\[
	{\Urj}\matT\xrrj=\xrj, \qquad
	{\Uj}\matT\xrrj=\xj.
\]
We show an illustration of the extended blocks in \Cref{fig:block_like}.
\begin{figure}[htbp]
	\centering
	\includegraphics[width=0.7\textwidth]{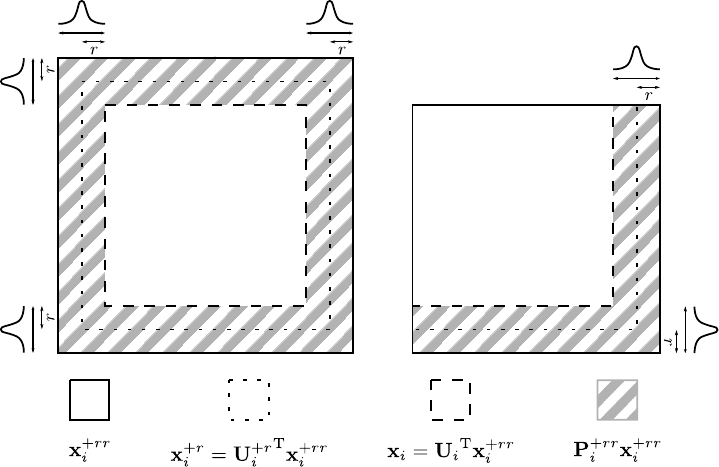}
	\caption{Examples of extended blocks for an interior block (left) and a corner block (right) in order to compute the local block likelihood.}
	\label{fig:block_like}
\end{figure}

Now we can evaluate the potential of the local block likelihood by
\begin{equation}\label{eq:block_like_0}
	l^{k,i-1}(\xrrj) = \frac{\lambda}{2} \| \yrj - {\Urj}\matT \Arrj \xrrj \|_2^2,
\end{equation}
where $\Arrj \in \mathbb{R}^{\qrrj \times \qrrj}$ is the convolution matrix with the correct dimensions for the extended block $\xrrj$.
Moreover, we let 
\begin{align}
	\Aj &= {\Urj}\matT \Arrj \Uj \label{eq:Aj} \\
	\by^{k,i-1} &= \yrj - {\Urj}\matT \Arrj \Prrj \xrrj \label{eq:byk}
\end{align}
where $\Prrj=\boldsymbol{I}-\Uj{\Uj}\matT\in\mathbb{R}^{\qrrj\times\qrrj}$ is an orthogonal projector, which projects $\xrrj$ onto the fixed pixels of the current update, see the hatched area in \Cref{fig:block_like}.
Then, we can formulate \cref{eq:block_like_0} as a function in $\xj$:
\begin{equation}\label{eq:block_like}
	l^{k,i-1}(\xj) = \frac{\lambda}{2} \| \by^{k,i-1} - \Aj \xj \|_2^2,
\end{equation}
and the gradient is easily obtained as
\begin{equation}\label{eq:block_like_grad}
	\nabla l^{k,i-1}(\bx_i) = -\lambda \Aj\matT  ( \by^{k,i-1} - \Aj \xj ).
\end{equation}

\subsection{Local block prior}
The set of pixels $\Phi_i$ on which the conditional \cref{eq:red_cond_blocks} depends is determined by the convolution. 
To compute the TV in $\varphi_\varepsilon^{k,i-1}$, we do not require all pixels in $\Phi_i$, but only the pixels of block $i$, extended by a frame of one pixel around it.
This is similar to the computation of the local block likelihood with $2r=1$, see \Cref{fig:block_like} again.
In the following, we denote this extended block by $\xonej\in\mathbb{R}^{\qonej}$.
Note that the extended block $\xonej$ can have different sizes $\qonej\in\{(m+1)^2, (m+1)(m+2), (m+2)^2 \}$, depending on whether it is located in a corner, at an edge or in the interior of the image.
Moreover, we define the component selection matrix $\Wj\in\mathbb{R}^{\qonej\times\qj}$, which allows for the following mapping between $\xj$ and $\xonej$:
$$
{\Wj}\matT\xonej=\xj.
$$
Now let 
\[
	\Dj = \Donej \Wj, \qquad
	\bj = \Donej \Ponej \xonej,
\]
where $\Donej\in\mathbb{R}^{\qonej\times\qonej}$ computes the differences in vertical ($l=v$) or horizontal ($l=h$) direction of the extended block $\xonej$.
Moreover, the orthogonal projector $\Ponej=\boldsymbol{I}-\Wj{\Wj}\matT\in\mathbb{R}^{\qonej \times \qonej}$ projects $\xonej$ onto the pixels, which are fixed during the update.
This is similar to the computation of the local block likelihood, see the hatched area in \Cref{fig:block_like} again.
Then, we can evaluate the potential of the local block prior by
\begin{equation}\label{eq:block_pr}
	\varphi_\varepsilon^{k,i-1}(\xj) = \delta \sum_\alpha \sqrt{ [\Dvj \xj + \bvj]_\alpha^2 + [\Dhj \xj + \bhj]_\alpha^2 + \varepsilon },
\end{equation}
and its gradient reads 
\begin{equation}
\label{eq:block_pr_grad}
\begin{split}
    \nabla \varphi_\varepsilon^{k,i-1}(\bx_i)
	=~& {\Dvj}\matT \Lambda(\xj)^{-1} (\Dvj \xj + \bvj) \\
	&+ {\Dhj}\matT \Lambda(\xj)^{-1} (\Dhj \xj + \bhj), \\
    \text{where} \quad \Lambda(\xj)= {\rm diag} &\Brac{ \sqrt{ [\Dvj \xj + \bvj]_\alpha^2 + [\Dhj \xj + \bhj]_\alpha^2 + \varepsilon } }.
\end{split}
\end{equation}

\begin{rem}
It is straightforward to extend the above computations to other types of difference operators (e.g., rank-deficient difference matrices, higher-order finite differences, etc.) as long as they exhibit the local structure, i.e., they act spatially local on the pixels. 
\end{rem}

\subsection{Local \& parallel MLwG algorithm}

Due to the local conditional dependencies in the posterior, we may be able to update several blocks in parallel during the loop in line \ref{algline:it_k} in \Cref{alg:MLwG}.
However, when updating block $i$, we have to fix all pixels within a frame of radius $2r$ around this block, and therefore, we can not update any of the neighboring blocks at the same time.
For an example, see \Cref{fig:upd_sched}, where the hatched area corresponds to the fixed blocks when updating block $\bx_6$.

We define a parallel updating scheme via the index sets $\mathcal{U}_l\subset[b],\,l=1,2,\dots$ with $\mathcal{U}_l \cap \mathcal{U}_p = \varnothing, ~ \forall l\neq p$, such that the blocks $\{i\,|\, i\in\mathcal{U}_l\}$ can be updated in parallel.
The choices $\mathcal{U}_l$ are not unique, but it holds $l\geq4$ for 2D images with square block partition.
In this paper, we use the minimal number of required updating sets $\mathcal{U}_l$, i.e., $l\in\{1,2,3,4\}$.
With this choice, the number of blocks which can be updated in parallel is
\begin{equation}\label{eq:num_par_upd}
	|\mathcal{U}_l| \geq \floor*{\frac{n}{2m}}^2,
\end{equation}
where we recall that $n$ is the side length of the image and $m$ the side length of the blocks.
An example of a parallel updating schedule is illustrated in \Cref{fig:upd_sched}.

In \Cref{alg:loc_TV}, we state a \emph{local \& parallel} version of the MLwG \Cref{alg:MLwG}.
That is, we consider \emph{parallel} updates in line \ref{algline:par_upd_loop}, and we evaluate the full conditionals in the block updates \emph{locally}, i.e., we use the expressions for the local block likelihood \cref{eq:block_like} and the local block prior \cref{eq:block_pr} in line \ref{algline:loc_upda}.
\begin{algorithm}[htbp]
	\caption{Local \& Parallel MLwG}
	\label{alg:loc_TV}
	\begin{algorithmic}[1]
		\REQUIRE Number of samples $N$, step size $\tau$, initial state $\bx^{0}$, partition into blocks $i=1,\dots,b$, sets of indices for parallel updating $\mathcal{U}_l$.
		\FOR{$k=1$ to $N$}
		\STATE Set $\bx^{k,0}\gets\bx^{k-1}$.
		\FOR{$l=1$ to $4$}\label{algline:seq_upd_loop}
		\FORALL[This loop can be done in parallel.]{$i\in\mathcal{U}_l$}\label{algline:par_upd_loop}
		\STATE Set $\bx^{k,i}\gets\bx^{k,i-1}$.
		\STATE Get $\xrrj$ and $\xonej$ from the current state $\bx^{k,i}$ (see \cref{fig:block_like} for example).%
		\STATE Propose a block candidate $\bz_i^{k,i}$ via \cref{eq:MALA_proposal} for the local block density \cref{eq:loc_block_dens}, with potentials given by \cref{eq:block_like,,eq:block_pr}. The gradients of the potentials are given by \cref{eq:block_like_grad,,eq:block_pr_grad}.\label{algline:loc_upda}
		\STATE Compute the acceptance probability $\alpha^{k,i}$ \cref{eq:ac_rate}.
		\STATE Simulate $u\sim\mathrm{Uniform}(0,1)$. 
		\IF{$\alpha^{k,i} > u$}
		\STATE Set $\bx_i^{k,i} \gets \bz_i^{k,i}$
		\ENDIF
		\ENDFOR
		\ENDFOR
		\ENDFOR
	\end{algorithmic}
\end{algorithm}

\section{Numerical examples}\label{sec:num_ex}
Both images in the following experiments are in grayscale and are normalized such that the pixel values are between 0 and 1. 
Note however, that we do not enforce the box constraint $[0, 1]$ on the pixel values in our experiments. 
In each sampling experiment, we compute 5 independent sample chains with $2000$ samples each and apply thinning by saving only every $200$-th sample to reduce correlation.
We check our sample chains for convergence by means of the potential scale reduction factor (PSRF) \cite{gelmanInferenceIterativeSimulation1992}. 
In brief, PSRF compares the within-variance with the in-between variance of the chains. 
Empirically, one considers sample chains to be converged if $\mathrm{PSRF}<1.1$. 
We use the Python package arviz \cite{Kumar2019} to compute the PSRF. 
With the same package, we compute the normalized effective sample size (nESS) and credibility intervals (CI), see, e.g., \cite{murphyMachineLearningProbabilistic2012} for definitions.
Our code is publicly available in \cite{Flock_2025}.
\subsection{Cameraman}
In this example, we first check the effect of different choices of the smoothing parameter $\varepsilon$ on the posterior density and sampling performance of MLwG.
Then, we illustrate the validity of the theoretical results from \Cref{sec:dim_ind_acc_conv}, namely the dimension-independent block acceptance and convergence rate of MLwG. 
Finally, we compare MLwG to MALA and show that the local \& parallel MLwG given in \Cref{alg:loc_TV} clearly outperforms MALA in increasing dimension in terms of sample quality and wall-clock time.

\subsubsection{Problem description}
We consider a blurred and noisy image, “cameraman”, with the size $512\times 512$ and show the “true” image on the left in \Cref{fig:compl_sect}.
The different image sections defined by the black frames and block partitions in the same image are required for the experiments in \Cref{sec:dim_ind_exp,,sec:comp_MALA_cam}. 
The data is synthetically obtained via the observation model \cref{eq:obs_model}, where $\boldsymbol{A}\in\mathbb{R}^{d\times d}$ corresponds to the discretization of a Gaussian blurring kernel with radius $8$ and standard deviation $8$. 
The noise is a realization of $\boldsymbol{\epsilon}\sim\mathcal{N}(\boldsymbol{0}, 0.01^2 \boldsymbol{I})$, and the degraded “cameraman” is shown on the right in \Cref{fig:compl_sect}.
\begin{figure}[htbp]
	\centering
	\includegraphics[scale = 0.85]{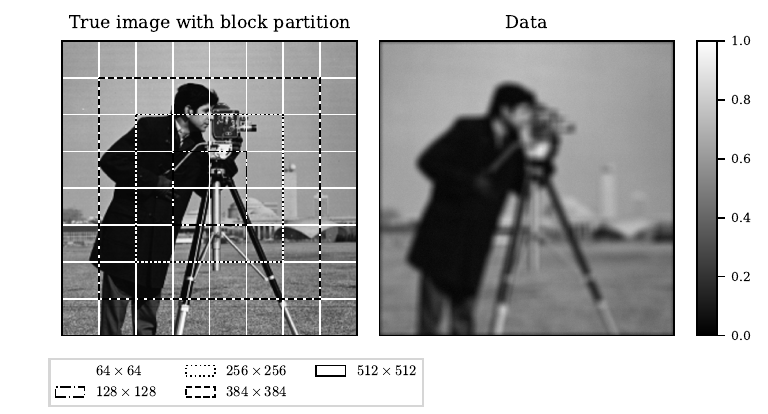}
	\caption{Left: True “cameraman” image and partition into deblurring problems of different sizes (black frames). 
		All sections are again partitioned into blocks of equal size $64\times 64$ (white frames). 
		Right: Data obtained via Gaussian blur and additive Gaussian noise.}
	\label{fig:compl_sect}
\end{figure}

We use the adaptive total variation approach in \cite{oliveiraAdaptiveTotalVariation2009} to determine the rate parameter $\delta$ in the TV prior \cref{eq:TV_prior}, and obtain $\delta=35.80$ for the $512\times 512$ image. 
We use this choice for all other problem sizes in \Cref{sec:dim_ind_exp,,sec:comp_MALA_cam} as well. 

\subsubsection{\texorpdfstring{Influence of the smoothing parameter $\boldsymbol{\varepsilon}$}{Influence of the smoothing parameter varepsilon}}
We compute MAP estimates for $\varepsilon\in\{10^{-3}, 10^{-5}, 10^{-7}\}$ with the majorization-minimization algorithm proposed in \cite{Figueiredo2006} and show the results in the left column of \Cref{fig:MAP_mean_CI_eps}.

We can see that the restoration from $\varepsilon=10^{-3}$ has smoother edges than the other two restorations which exhibit the typical cartoon-like structure of TV-regularized images.
Further, the difference between the results from $\varepsilon=10^{-5}$ and $\varepsilon=10^{-7}$ is hardly visible.

\begin{figure}
	\centering
	\includegraphics{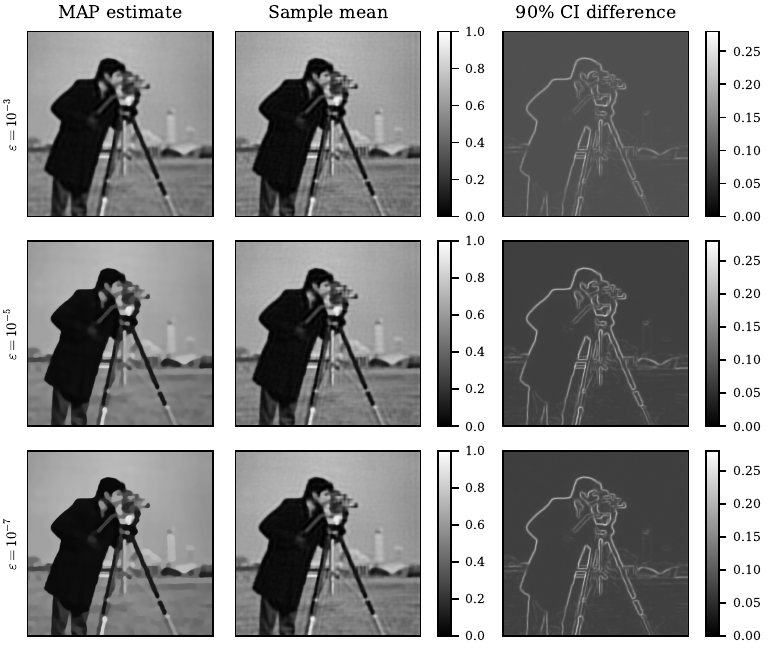}
	\caption{MAP estimates, sample means, and widths of the 90\% sample CIs for $\varepsilon\in\{10^{-3}, 10^{-5}, 10^{-7}\}$.
		The sampling results are obtained with the local \& parallel MLwG given in \Cref{alg:loc_TV}.}
	\label{fig:MAP_mean_CI_eps}
\end{figure}

We now run our local \& parallel MLwG (\Cref{alg:loc_TV}) for $\varepsilon\in\{10^{-3}, 10^{-5}, 10^{-7}\}$ with a diminishing step size adaptation during burn-in \cite{marshallAdaptiveApproachLangevin2012b} targeting an acceptance rate of $0.547$ in each block.\
\footnote{We note that the diminishing step size adaptation proposed in \cite{marshallAdaptiveApproachLangevin2012b} is performed during the whole sampling process. 
In our experiments, we stop the adaptation after a specified number of burn-in samples.}
We chose this target acceptance rate due to the results in \cite{robertsOptimalScalingVarious2001a}.
We show the sample means and widths of the 90\% sample CIs in \Cref{fig:MAP_mean_CI_eps}.
Here, the sample means of $\varepsilon=10^{-5}$ and $\varepsilon=10^{-7}$ are visually more favorable than their corresponding MAP estimates. 
In contrast, the result of $\varepsilon=10^{-3}$ contains visible artifacts.

Moreover, the 90\% sample CI difference is in general wider for $\varepsilon=10^{-3}$ than for $\varepsilon=10^{-5}$ and $\varepsilon=10^{-7}$.
However, on the edges, the width of the 90\% sample CIs are rather similar.
We also list some results about the sample chains in \Cref{tab:MALA_eps}.
Here, we note that $\varepsilon=10^{-3}$ allows for a significantly larger mean step size in comparison to $\varepsilon=10^{-5}$ and $\varepsilon=10^{-7}$.
This results in less correlated samples, which is reflected in a larger nESS.

We conclude that relatively large values of $\varepsilon$ make the posterior density smoother, allowing for larger step sizes and thus making the sampling more efficient in terms of nESS.
However, at the same time, the results can be visually significantly different compared to choices of small $\varepsilon$, which yield sharper edges in the MAP estimate and the mean. 
Based on these observations, and since the results for $\varepsilon=10^{-5}$ and $\varepsilon=10^{-7}$ are very close, we fix $\varepsilon=10^{-5}$ in the remaining experiments.
\begin{table}[htbp]
	\footnotesize
	\caption{Sampling results of MLwG for $\varepsilon\in\{10^{-3}, 10^{-5}, 10^{-7}\}$. 
	   The shown min nESS is the pixel-wise minimum of the mean nESS which is taken over the $5$ chains.
		The shown step size $\tau$ and acceptance rate $\alpha$ are the means taken over all blocks and the $5$ chains.
		The maximum and median PSRF are with respect to the pixels.}
	\label{tab:MALA_eps}
	\begin{center}
		\begin{tabular}{ccccccc} 
			\toprule
			$\boldsymbol{\varepsilon}$ &\textbf{min nESS [\%]} &$\boldsymbol{\tau}$ \textbf{[10\textsuperscript{-6}]} &$\boldsymbol{\alpha}$ \textbf{[\%]} &\textbf{max PSRF} &\textbf{median PSRF}\\ 
			\midrule
			$10^{-3}$ &13.3 &25.8 &54.7 &1.01 &1.00 \\
			$10^{-5}$ &3.2 &7.5 &54.4 &1.03 &1.00 \\
			$10^{-7}$ &2.1 &5.6 &54.3 &1.04 &1.00 \\ 
			\bottomrule
		\end{tabular}
	\end{center}
\end{table}
\subsubsection{Dimension-independent block acceptance rate}\label{sec:dim_ind_exp}
To test the\newline dimension-independent block acceptance rate in \cref{prop:ac_rate_bound}, 
we partition the original $512\times 512$ image into 4 sections of sizes $128\times 128$, $256\times 256$, $384\times 384$, and $512\times 512$.
Furthermore, each section is partitioned into blocks of equal size $64 \times 64$. 
Thus, the number of blocks in the sections of sizes $128\times 128$, $256\times 256$, $384\times 384$, and $512\times 512$ are $4$, $16$, $36$, and $64$, respectively. 
The 4 deblurring problems are shown on the left in \cref{fig:compl_sect}.

We run the local \& parallel MLwG given in \Cref{alg:loc_TV} with a step size of $\tau=7.44\times 10^{-6}$ on the 4 deblurring problems with different sizes.
The step size is taken from a pilot run on the $512\times 512$ problem by targeting an acceptance rate of $0.547$ in each block and then taking the average of all block step sizes. 
For all problem sizes, we use a burn-in period of $31,250$ samples.
We plot the acceptance rate for each block in \cref{fig:loc_acc} and see that the block acceptance rates are indeed dimension-independent. 

\begin{figure}[htbp]
	\centering
	\includegraphics[scale=0.85]{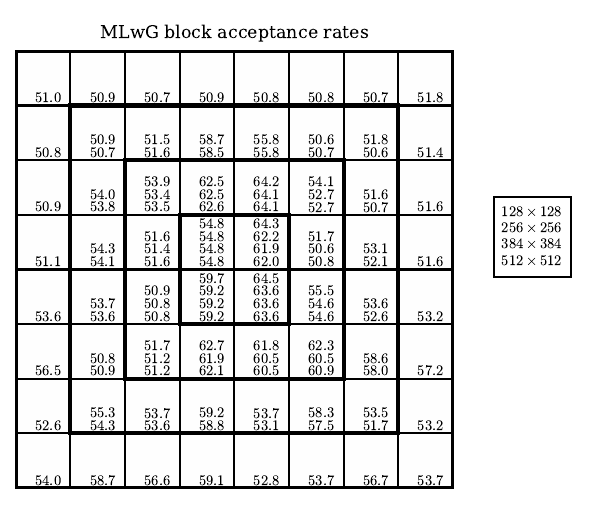}
	\caption{Block acceptance rates of MLwG for the different problem sizes in \%. 
		The acceptance rates are listed according to the problem sizes in the order shown on the right.} 
\label{fig:loc_acc}
\end{figure}
\subsubsection{Comparison to MALA}\label{sec:comp_MALA_cam}
In this subsection, we compare the performance of our local \& parallel MLwG (\Cref{alg:loc_TV}) with MALA. 
For MALA, we use again the diminishing step size adaptation from \cite{marshallAdaptiveApproachLangevin2012b} during burn-in, where we target an acceptance rate of $0.547$.
The numbers of burn-in samples are listed in \Cref{tab:MLwG_MALA} and are chosen such that they increase linearly with the problem size. 
For MLwG, we use the same setting as in the previous tests.

We compare the sampling performance of MALA and MLwG in \cref{tab:MLwG_MALA}.
In general, MLwG yields much larger nESS than MALA because it allows for a larger step size.
In MALA, the proposal is global, and it is known that one should take $h= O(d^{-1/3})$ to ensure a non-degenerate acceptance rate \cite{MR1625691}. 
In contrast, in MLwG, the proposal is local, and the step size can be taken independent of the total dimension, see \Cref{prop:ac_rate_bound}.
Furthermore, the nESS of MLwG becomes even larger as the problem size increases. 
We attribute this to the diminishing constraining effect of the boundary condition associated with the convolution operator on the inner blocks as the dimension increases. 
In addition, we note that for the given burn-in, MALA does not converge for the problem sizes $384\times 384$ and $512\times 512$, since the corresponding $\mathrm{max}~\mathrm{PSRF}>1.1$.
\begin{table}[htbp]
\footnotesize
\caption{Comparison of the local \& parallel  MLwG and MALA. 
    The printed min nESS is the pixel-wise minimum of the mean nESS which is taken over the $5$ chains.
    For MLwG, the shown step size $\tau$ and acceptance rate $\alpha$ are the means taken over all blocks and the $5$ chains.
	The maximum PSRF is with respect to the pixels.}
\label{tab:MLwG_MALA}
\begin{center}
	\begin{tabular}{lccccc} 
		\toprule
		\bf{Problem size} & &\bf 128\texttimes128 &\bf 256\texttimes256 &\bf 384\texttimes384 &\bf 512\texttimes512 \\
		\midrule
		\multirow{2}{*}{\textbf{min nESS [\%]}} 
            &MLwG &2.6&2.8&2.9&3.0\\
            &MALA&1.4&0.8&0.6&0.4\\
		\midrule
		\multirow{2}{*}{$\boldsymbol{\tau}$ \textbf{[10\textsuperscript{-6}]}} 
		&MLwG&7.4&7.4&7.4&7.4\\
		&MALA&4.8&2.5&1.8&1.4\\
		\midrule
		\multirow{2}{*}{$\boldsymbol{\alpha}$ \textbf{[\%]}}
		&MLwG&60.8&57.7&55.5&54.3\\
		&MALA&54.0&54.3&54.7&55.0\\
		\midrule
		\multirow{2}{*}{\textbf{burn-in [10\textsuperscript{3}]}}
		&MLwG&31.250&31.250&31.250&31.250\\
		&MALA&125.000&500.000&1125.000&2000.000\\
		\midrule
		\multirow{2}{*}{\textbf{max PSRF}}
		&MLwG&1.03&1.03&1.03&1.03\\
		&MALA&1.06&1.08&1.20&1.20\\
		\bottomrule
	\end{tabular}
\end{center}
\end{table}

Notice that the results from \cref{tab:MLwG_MALA} also validate the dimension-independent convergence rate in \cref{prop:conv} of MLwG.
This is because MLwG produces for all problem sizes and with the same burn-in converged chains with roughly constant PSRF.
In contrast, MALA requires significantly more burn-in with increasing dimension.
We note that despite the clear differences in sampling performance between MLwG and MALA, we do not see clear differences in the sample mean or the widths of the $90\%$ CIs. 
Therefore, we do not include these results here.

Finally, we compare the wall-clock and CPU time of the sample chains of the local \& parallel MLwG and MALA. 
All chains are run on the same hardware, specifically, Intel\textsuperscript{\textregistered} Xeon\textsuperscript{\textregistered} E5-2650 v4 processors, which are installed on a high performance computing cluster.
Furthermore, we use the optimal number of cores for MLwG, such that all blocks with indices $i\in\mathcal{U}_l$ during loop $l$ in line \ref{algline:par_upd_loop} in \Cref{alg:loc_TV} can be updated in parallel.

We show the computing times in seconds per $1000$ samples in \cref{fig:comp_times} and observe that the wall-clock time of MLwG remains almost constant and does not increase with the problem dimension.
This is because the main computational effort of updating the $64\times 64$ blocks on each core remains constant and only more time is required for handling the increasing number of cores by the main process.
For small problem sizes, the wall-clock time of MLwG is longer than that of MALA, because of the overhead of the parallelized implementation and the additional convolutions of fixed pixels in the local block likelihoods.
However, since several updates are run in parallel in MLwG, its wall-clock time is eventually shorter than that of MALA, see the time for problem size $512\times512$.
Note that the total wall-clock time of MALA is actually significantly larger, since it requires much more burn-in.
The benefits of MLwG obviously come at the cost of CPU time, which increases linearly with the number of cores.

\begin{figure}[htbp]
\centering
\includegraphics[scale=0.8]{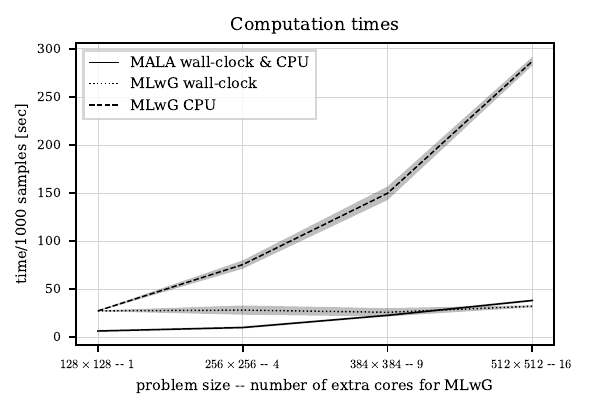}
\caption{Mean wall-clock and CPU times of MLwG and MALA. 
	For MLwG, we show the mean plus/minus the standard deviation by means of the shaded area. 
	The wall-clock and CPU time of MALA are approximately equal and are therefore not displayed separately. 
	For MLwG, we used the number of cores indicated in the x-tick labels plus one additional core to handle the main process.}
\label{fig:comp_times}
\end{figure}
\subsection{House} 
In this section, we use another test image with a different blurring kernel to compare our local \& parallel MLwG (\Cref{alg:loc_TV}) with MALA.
\subsubsection{Problem description}
The degraded image is synthetically obtained via the observation model \cref{eq:obs_model}, where $\boldsymbol{A}\in\mathbb{R}^{d\times d}$ is obtained through the discrete PSF of a motion blur kernel with the length $17$ and the angle $45^\circ$.
The noise is a realization of $\boldsymbol{\epsilon}\sim\mathcal{N}(\boldsymbol{0}, 0.01^2 \boldsymbol{I})$. 
The true and the degraded image of “house” with the size $512\times 512$ are shown in the first row of \cref{fig:house}.
\begin{figure}[htbp]
\centering
\includegraphics[scale=0.8]{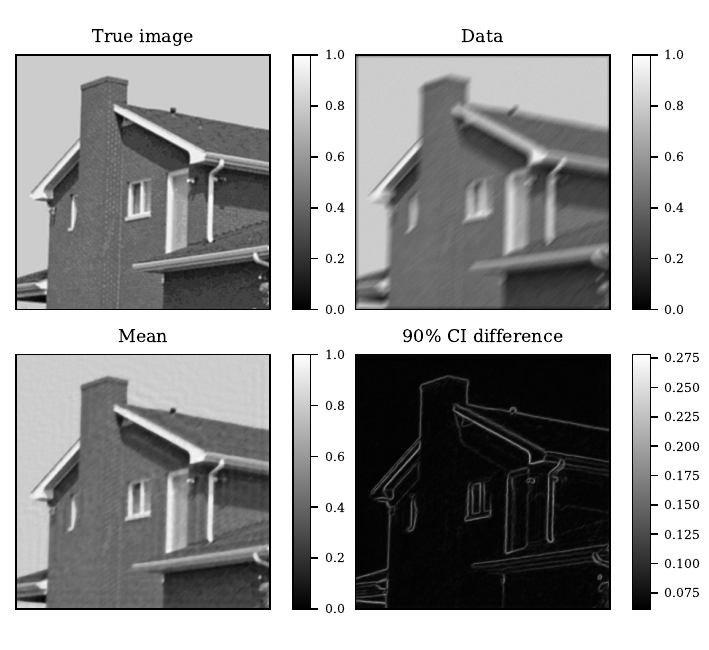}
\caption{The true image, the degraded image, the posterior sample mean, and the width of the 90\% sample CI.}
\label{fig:house}
\end{figure}
We use the adaptive total variation approach in \cite{oliveiraAdaptiveTotalVariation2009} to determine the rate parameter $\delta$ in the TV prior \cref{eq:TV_prior} and obtain $\delta=34.83$. 

\subsubsection{Posterior sampling via MLwG and MALA}
Again, we apply both the local \& parallel MLwG and MALA to sample the smoothed posterior defined in \cref{eq:smooth_pos_expl}.
In both methods, we target an acceptance rate of $0.547$ using again the method from \cite{marshallAdaptiveApproachLangevin2012b}, where we adapt the step size during burn-in.
In MLwG, we adapt the step size individually in each block.
In the second row of \Cref{fig:house}, we show the posterior mean and the width of the 90\% CI obtained via the samples from MLwG. The restoration results from MALA are neglected, since they are visually identical to the ones from MLwG. 
However, the data about the sample chains that we show in \Cref{tab:MLwG_MALA_house} reveals once again the superior performance of MLwG over MALA.
Similar as observed in the previous test, MLwG allows for a significantly larger step size, which leads to faster convergence and less correlated samples in terms of nESS.
\begin{table}[htbp]
\footnotesize
\caption{
Sampling data of MLwG and MALA for the “house” image.
The printed min nESS is the pixel-wise minimum of the mean nESS which is taken over the $5$ chains.
For MLwG, the shown step size $\tau$ and acceptance rate $\alpha$ are the means taken over all blocks and the $5$ chains.
The maximum PSRF is with respect to the pixels.
}
\label{tab:MLwG_MALA_house}
\begin{center}
	\begin{tabular}{lccccc} 
		\toprule
		&\textbf{min nESS [\%]} &$\boldsymbol{\tau}$ \textbf{[10\textsuperscript{-6}]} &$\boldsymbol{\alpha}$ \textbf{[\%]} &\textbf{burn-in [10\textsuperscript{3}]} &\textbf{max PSRF} \\
		\midrule
        MLwG&4.5&7.4&54.3&31.250&1.02\\
        MALA&0.8&1.4&55.9&2000.000&1.09\\
		\bottomrule
	\end{tabular}
\end{center}
\end{table}

\section{Conclusions}\label{sec:conc}
Uncertainty quantification in imaging problems is usually a difficult task due to the high dimensionality of images.
For image deblurring with TV prior, we presented a dimension-independent MLwG sampling algorithm. 
By exploiting the sparse conditional structure of the posterior, the proposed algorithm has dimension-independent block acceptance and convergence rates, which have been theoretically proven and numerically validated.  
To enable the use of the MALA proposal density, we used a smooth approximation of the TV prior.
We showed that the introduced error is uniformly distributed over the pixels and thus is dimension-independent.
Moreover, through numerical studies, we found that the smoothed posterior converges quickly to the exact posterior and yields feasible results for uncertainty quantification.

Various extensions of the local MALA-within-Gibbs method can be explored for imaging problems. 
For instance, the TV prior can be generalized to other difference operator-based priors that are of practical interest. 
On the implementation side, the local and parallel sampling nature of the method enables efficient GPU acceleration. 
We leave these considerations for future work.

\section*{Acknowledgments}
The work of RF and YD has been supported by a Villum Investigator grant (no.\ 25893) from The Villum Foundation. 
The work of SL has been partially funded by Singapore MOE grant A-8000459-00-00.
The work of XT has been funded by Singapore MOE grants A-0004263-00-00 and A-8000459-00-00.
We thank the anonymous reviewers for their insightful questions and comments, which contributed to a significant improvement of the manuscript.

\bibliographystyle{siamplain}
\bibliography{SISC-bib}

\appendix
\section{Proofs}

\subsection{Dimension-independent approximation error}  \label{app:ApproxErr}
We mainly use Stein's method to prove \cref{thm:TV_ApproxErr}. Two technical lemmas are provided in \Cref{app:lems}. 

\subsubsection{\texorpdfstring{Proof of \cref{thm:TV_ApproxErr}}{Proof of Theorem 3.2}}
Due to the Kantorovich duality, the\newline Wasserstein-$1$ distance of $\pi_i$ and $\pi_{\varepsilon,i}$ can be written as
\[
W_1(\pi_i,\pi_{\varepsilon,i}) = \sup_{\phi_0 \in \Lip_1(\mR^q)} \int \phi_0(\bx_i) \Brac{ \pi_i (\bx_i) - \pi_{\varepsilon,i} (\bx_i) }\mdd \bx_i, 
\]
where $\Lip_1(\mR^q)$ denotes the $1$-Lipschitz function class. For any $\phi_0 \in \Lip_1$, denote
\[
\phi(\bx_i) = \phi_0(\bx_i) - \int \phi_0(\bx_i) \pi_{\varepsilon,i} (\bx_i) \mdd \bx_i = \phi_0(\bx_i) - \int \phi_0(\bx_i) \pi_{\varepsilon} (\bx) \mdd \bx.
\]
Then, $\phi \in \Lip_1^0(\pi_{\varepsilon,i})$ is the class of $1$-Lipschitz functions that are mean-zero w.r.t. $\pi_{\varepsilon,i}$: 
\begin{equation}    \label{eq:Lip10}
\Lip_1^0(\pi_{\varepsilon,i}) := \left\{ \phi \in \Lip_1(\mR^q) : \int \phi(\bx_i) \pi_{\varepsilon,i} (\bx_i)\mdd \bx_i = 0 \right\}. 
\end{equation}
Then, by definition,   
\begin{align*}
\int \phi_0(\bx_i) \Brac{ \pi_i (\bx_i) - \pi_{\varepsilon,i} (\bx_i) }\mdd \bx_i =~& \int \phi(\bx_i) \Brac{ \pi_i (\bx_i) - \pi_{\varepsilon,i} (\bx_i) }\mdd \bx_i \\
=~& \int \phi(\bx_i) \pi_i (\bx_i) \mdd \bx_i = \int \phi(\bx_i) \pi (\bx) \mdd \bx .
\end{align*}
Given $\phi(\bx_i)$, consider the Poisson equation for $u(\bx)$:
\begin{equation}    \label{eq:PoisEqn}
\Delta u(\bx) + \nabla \log \pi_\varepsilon (\bx) \cdot \nabla u(\bx) = \phi(\bx_i). 
\end{equation}
By \cref{lem:PoisGrad}, the solution exists and satisfies the gradient estimate
\[
\sum_j \norm{ \nabla_j u }_{L^\infty} \leq 2 \lambda^{-1} c^{-1}.
\]
Then by integration by parts,
\begin{align*}
\int \phi(\bx_i) \pi (\bx) \mdd \bx =~& \int \Brac{ - \nabla u(\bx) \cdot \nabla \log \pi(\bx) + \nabla \log \pi_\varepsilon (\bx) \cdot \nabla u(\bx) } \pi (\bx) \mdd \bx \\ 
\leq~& \sum_{j=1}^b \norm{ \nabla_j u }_{L^\infty} \cdot \max_j \norm{ \nabla_j \log \pi_\varepsilon -  \nabla_j \log \pi }_{L^1(\pi)}.
\end{align*}
Combining the above two inequalities, we have
\begin{equation}	\label{eq:Pf_ctr1}
\begin{split}
	\max_i W_1(\pi_i,\pi_{\varepsilon,i} ) =~& \max_i \sup_{\phi \in \Lip_1^0(\pi_{\varepsilon,i})} \int \phi(\bx_i) \pi (\bx) \mdd \bx \\
	\leq~& 2 \lambda^{-1} c^{-1} \max_j \norm{ \nabla_j \log \pi_\varepsilon -  \nabla_j \log \pi }_{L^1(\pi)}.
\end{split}
\end{equation}
Note here, $\nabla_j \log \pi$ is not defined pointwise, but it suffices to require that $\nabla_j \log \pi \in L^1(\pi)$. It suffices to control the right hand side of \cref{eq:Pf_ctr1}. By definition, 
\begin{align*}
& \nabla_j \log \pi(\bx) -  \nabla_j \log \pi_\varepsilon(\bx) = \nabla_j \varphi_\varepsilon(\bx) - \nabla_j \varphi_0(\bx) \\
=~& \delta \sum_{\alpha = 1}^d \Rectbrac{ \Brac{ (\bD_\alpha^{(v)} \bx)^2+ (\bD_\alpha^{(h)} \bx)^2 + \varepsilon }^{-1/2} - \Brac{ (\bD_\alpha^{(v)} \bx)^2+ (\bD_\alpha^{(h)} \bx)^2 }^{-1/2} } \\
&\qquad \cdot\Brac{ \Brac{\bD_\alpha^{(v)} \bx} \Brac{\bD_{\alpha,j}^{(v)}}\matT + \Brac{\bD_\alpha^{(h)} \bx} \Brac{\bD_{\alpha,j}^{(h)}}\matT } =: \delta \sum_{\alpha = 1}^d \bv_\alpha(\bx). 
\end{align*}
Here, we denote $\bD_{\alpha,j}^{(v)} \in \mR^{1\times q}$ as the $j$-th block of $\bD_{\alpha} \in \mR^{1\times d}$ and similar for $\bD_{\alpha,j}^{(h)}$. To control $\norm{\bv_\alpha(\bx)}_2$, notice
\begin{align*}
\norm{ \Brac{\bD_\alpha^{(v)} \bx} \Brac{\bD_{\alpha,j}^{(v)}}\matT + \Brac{\bD_\alpha^{(h)} \bx} \Brac{\bD_{\alpha,j}^{(h)}}\matT }_2 \leq | \bD_\alpha^{(v)} \bx | \normo{ \bD_{\alpha,j}^{(v)} }_2 + | \bD_\alpha^{(h)} \bx | \normo{\bD_{\alpha,j}^{(h)}}_2 & \\
\leq \sqrt{2} \Brac{ | \bD_\alpha^{(v)} \bx | + | \bD_\alpha^{(h)} \bx | } \leq 2 \Brac{ (\bD_\alpha^{(v)} \bx)^2+ (\bD_\alpha^{(h)} \bx)^2 }^{1/2} =: 2 \norm{\bD_\alpha \bx}_2 &.
\end{align*}
Here, we denote $\bD_\alpha \bx = ( \bD_\alpha^{(v)} \bx, \bD_\alpha^{(h)} \bx) \in \mR^2$. Notice $ \bD_{\alpha,j}^{(v)} = \bD_{\alpha,j}^{(h)}  = 0 $ if $\alpha$ does not live in block $j$ or a neighbor of block $j$, so that there are at most $(m+2)^2$ indices $\alpha$ s.t.\ $\bv_\alpha(\bx)$ is nonzero. Therefore, 
\begin{equation}    \label{eq:pf_ref2}
\begin{split}
	\| \nabla_j \log \pi(\bx) -  \nabla_j \log &\pi_\varepsilon(\bx) \|_{L^1(\pi)} \leq \delta \cdot \mE_{\bx \sim \pi} \sum_{\alpha : \bv_\alpha(\bx) \neq 0} \norm{ \bv_\alpha(\bx) }_2 \\
	\leq~& 2 \delta (m+2)^2 \max_\alpha \mE_{\bx \sim \pi} \Rectbrac{ 1 - \frac{  \norm{\bD_\alpha \bx}_2  }{ \sqrt{   \norm{\bD_\alpha \bx}_2^2+\varepsilon} } } . 
\end{split}
\end{equation}
Denote the function $I_\alpha(t) := \mE_{\bx \sim \pi} \Rectbrac{ 1 - \frac{ \norm{\bD_\alpha \bx}_2 }{\sqrt{ \norm{\bD_\alpha \bx}_2^2 + t^2} } }$. Then $I_\alpha(0) = 0$ and 
\begin{align*}
I_\alpha'(t) =~& \mE_{\bx \sim \pi} \Rectbrac{ \frac{ \norm{\bD_\alpha \bx}_2 t}{\Brac{ \norm{\bD_\alpha \bx}_2^2 + t^2}^{3/2}} } \leq \mE_{\bx \sim \pi} \Rectbrac{ \frac{ \norm{\bD_\alpha \bx}_2^2 + t^2}{2 \Brac{ \norm{\bD_\alpha \bx}_2^2 + t^2}^{3/2}} }  \\
\leq~& \mE_{\bx \sim \pi} \Rectbrac{ \frac{1}{2 \Brac{ \norm{\bD_\alpha \bx}_2^2 + t^2}^{1/2}} } \leq \frac{1}{2} \mE_{\bx \sim \pi} \norm{\bD_\alpha \bx}_2^{-1}. 
\end{align*}
By \cref{lem:Inv_Ctrl}, there exists a dimension-independent constant $C_\pi$ s.t. 
\[
    \max_\alpha \mE_{\bx \sim \pi} \norm{\bD_\alpha \bx}_2^{-1} \leq C_\pi.
\]
This implies that $I_\alpha'(t) \leq C_\pi/2 \Rightarrow I_\alpha(t) \leq C_\pi t/ 2$. So that
\[
\mE_{\bx \sim \pi} \Rectbrac{ 1 - \frac{ \norm{\bD_\alpha \bx}_2  }{ \sqrt{ \norm{\bD_\alpha \bx}_2^2+\varepsilon} } } \leq \frac{1}{2} C_\pi \varepsilon^{1/2}. 
\]
Substituting this into \cref{eq:pf_ref2}, we have
\[
\norm{ \nabla_j \log \pi(\bx) -  \nabla_j \log \pi_\varepsilon(\bx) }_{L^1(\pi)} \leq C_\pi \delta (m+2)^2 \varepsilon^{1/2}.
\]
Finally, substituting this into \cref{eq:Pf_ctr1}, we have
\[
\max_i W_1(\pi_i,\pi_{\varepsilon,i}) \leq 2 \lambda^{-1} c^{-1} \cdot C_\pi \delta (m+2)^2 \varepsilon^{1/2}.
\]
Notice we take $ \frac{\lambda}{\delta} \geq \frac{64}{c \sqrt{\varepsilon}} $ and $b>1$. The conclusion follows easily from
\[
2 \lambda^{-1} c^{-1} \cdot C_\pi \delta (m+2)^2 \varepsilon^{1/2} \leq C_\pi \cdot \frac{(m+2)^2}{32} \cdot \varepsilon =: C \varepsilon
\]
for some dimension-independent constant $C$.

\subsubsection{Some technical lemmas}   \label{app:lems}
\begin{lemma}    \label{lem:PoisGrad}
Under the assumptions in \cref{thm:TV_ApproxErr}, the solution $u(\bx)$ to the Poisson equation \cref{eq:PoisEqn} exists, unique up to a constant, and satisfies the gradient estimate
\begin{equation}    \label{eq:PoisGrad}
	\sum_j \norm{ \nabla_j u }_{L^\infty} \leq 2 \lambda^{-1} c^{-1}. 
\end{equation}
\end{lemma}

\begin{proof}
Denote the operator
\[
\mcL u(\bx) := \Delta u(\bx) + \nabla \log \pi_\varepsilon (\bx) \cdot \nabla u(\bx). 
\]
For simplicity, we still denote $\mcL\bv(\bx)$ when $\bv$ is a vector-valued function
\[
\mcL \bv(\bx) := ( \mcL \bv_1(\bx), \dots \mcL \bv_b(\bx) ) \matT. 
\]
It suffices to prove \cref{eq:PoisGrad} for $\phi \in C^1 \cap \Lip_1^0(\pi_{\varepsilon,i}) $. Note this space is dense in $\Lip_1^0(\pi_{\varepsilon,i})$, so for general $\phi \in \Lip_1(\pi_{\varepsilon,i})$, we can take a sequence of $\phi_k \in C^1 \cap \Lip_1^0(\pi_{\varepsilon,i})$ that converges to $\phi$. \cref{eq:PoisGrad} holds uniformly for $\phi_k$, so that passing the limit shows that it holds for any $\phi \in \Lip_1^0(\pi_{\varepsilon,i})$. Now fix any $\phi\in C^1 \cap \Lip_1^0(\pi_{\varepsilon,i})$. It is straightforward to verify by standard elliptic theory that the solution exists (up to a constant) and $u \in C^3$. Taking the gradient w.r.t.\ $\bx_j$ in \cref{eq:PoisEqn}, we obtain
\begin{equation}	\label{eq:pf_blk_grad}
	\nabla_j \phi (\bx_i) = \mcL(\nabla_j u) (\bx) + \sum_{k} \nabla_{jk}^2 \log \pi_\varepsilon (\bx) \nabla_k u(\bx). 
\end{equation}
Recall by definition \cref{eq:smooth_pos_expl}, $\pi_\varepsilon(\bx) \propto \exp ( - l(\bx) - \varphi_\varepsilon(\bx) )$. Direct computation shows
\[
\nabla^2 \log \pi_\varepsilon (\bx) = - \nabla^2 l(\bx) - \nabla^2 \varphi_\varepsilon(\bx) = - \lambda \bA\matT \bA - \nabla^2 \varphi_\varepsilon (\bx),
\]
where the $(i,j)$-th subblock of $\nabla^2 \varphi_\varepsilon (\bx)$ is given by
\begin{equation}	\label{eq:phi_Hess}
	\begin{split}
		&\nabla_{ij}^2 \varphi_\varepsilon (\bx) = \delta \sum_{\alpha=1}^d \Brac{ (\bD_\alpha^{(v)} \bx)^2+ (\bD_\alpha^{(h)} \bx)^2 + \varepsilon }^{-3/2} \\
		& \cdot \Big[ \Brac{ (\bD_\alpha^{(v)} \bx) \bD_{\alpha,i}^{(v)} - (\bD_\alpha^{(h)} \bx) \bD_{\alpha,i}^{(h)} } \matT \Brac{ (\bD_\alpha^{(v)} \bx) \bD_{\alpha,j}^{(v)} - (\bD_\alpha^{(h)} \bx) \bD_{\alpha,j}^{(h)} } \\
		& \quad + \varepsilon \Brac{ (\bD_{\alpha,i}^{(v)})\matT \bD_{\alpha,j}^{(v)} + (\bD_{\alpha,i}^{(h)})\matT \bD_{\alpha,j}^{(h)} } \Big]. 
	\end{split}
\end{equation}
For simplicity, denote $H(\bx) = \nabla^2 \varphi_\varepsilon(\bx)$. Multiplying \cref{eq:pf_blk_grad} from left by $ - \Brac{\nabla_j u (\bx)}\matT$, 
\begin{equation}	\label{eq:pf_ip_grad}
	\begin{split}
		- \Brac{\nabla_j u (\bx)}\matT \nabla_j \phi (\bx_i) +& \Brac{\nabla_j u (\bx)}\matT \mcL(\nabla_j u) (\bx) \\
		=~& \sum_k \Brac{\nabla_j u (\bx)}\matT ( \lambda \bA\matT \bA + H (\bx))_{jk} \nabla_k u(\bx).
	\end{split}
\end{equation}
Since $\bA\matT \bA$ is $c$-diagonal block dominant with matrix $M\in\mR^{b\times b}$, and $H_{jj}$ is positive definite, since $\varphi_\varepsilon$ is convex, we have 
\begin{align*}
	& \sum_k \Brac{\nabla_j u (\bx)}\matT (\lambda \bA\matT \bA + H(\bx) )_{jk} \nabla_k u(\bx) \\
	=~& \Brac{\nabla_j u (\bx)}\matT (\lambda \bA\matT \bA + H(\bx) )_{jj} \nabla_j u(\bx) + \sum_{k\neq j} \Brac{\nabla_j u (\bx)}\matT (\lambda \bA\matT \bA + H(\bx) )_{jk} \nabla_k u(\bx) \\
	\geq~& \lambda M_{jj} \normo{ \nabla_j u(\bx) }_2^2 - \lambda \sum_{k\neq j} ( M_{jk} +  \normo{ H_{jk} (\bx) }_2 ) \normo{ \nabla_j u(\bx) }_2 \normo{ \nabla_k u(\bx) }_2. 
\end{align*}
Next we control $\normo{ H_{jk} (\bx) }_2$. Note $\bD^{(v)}_{\alpha,i}=0$ if $\alpha$ does not live in block $i$ or a vertical neighbor of block $i$. Similarly for $\bD^{(h)}_{\alpha,i}$. Therefore, $ H_{ij} (\bx) = 0$ if $i$ and $j$ are not neighbors. For vertical neighbors $i,j$, there are exactly $2m$ boundary indices $\alpha$ s.t.\ $ \Brac{\bD^{(v)}_{\alpha,i}}\matT \bD^{(v)}_{\alpha,j} \neq 0$, and $\norm{ \Brac{\bD^{(v)}_{\alpha,i}}\matT \bD^{(v)}_{\alpha,j} }_2 \leq 2$. Moreover, 
\begin{align*}
	\norm{ \Brac{ (\bD_\alpha^{(v)} \bx) \bD_{\alpha,i}^{(v)} - (\bD_\alpha^{(h)} \bx) \bD_{\alpha,i}^{(h)} } \matT \Brac{ (\bD_\alpha^{(v)} \bx) \bD_{\alpha,j}^{(v)} - (\bD_\alpha^{(h)} \bx) \bD_{\alpha,j}^{(h)} } }_2 &\\
	\leq \Brac{ \sqrt{2} |\bD_\alpha^{(v)} \bx| + \sqrt{2} |\bD_\alpha^{(h)} \bx| }^2 \leq 4 \Brac{ (\bD_\alpha^{(v)} \bx)^2 + (\bD_\alpha^{(h)} \bx )^2 } &. 
\end{align*}
Therefore, recall \cref{eq:phi_Hess}. It holds that
\begin{equation}    \label{eq:phi_Hess_bound}
	\begin{split}
		\norm{ H_{ij} (\bx) }_2 \leq~& \delta \sum_{ \alpha \text{ in boundary} } \Brac{ (\bD_\alpha^{(v)} \bx)^2+ (\bD_\alpha^{(h)} \bx)^2 + \varepsilon }^{-3/2} \\
		& \cdot \Rectbrac{ 4 \Brac{ (\bD_\alpha^{(v)} \bx)^2 + (\bD_\alpha^{(h)} \bx )^2 } + 4 \varepsilon } \\
		\leq~& 4 \delta \sum_{ \alpha \text{ in boundary} } \Brac{ (\bD_\alpha^{(v)} \bx)^2+ (\bD_\alpha^{(h)} \bx)^2 + \varepsilon }^{-1/2} \leq 8 m \delta \varepsilon^{-1/2}.
	\end{split}
\end{equation}
A similar result holds when $i,j$ are horizontal neighbors. Also notice $H_{jk}(\bx) = 0$ when $j \not\sim k$, where we denote $j\sim k$ if $j\neq k$ are neighboring blocks. Now substitute the above inequalities into \cref{eq:pf_ip_grad}. We have 
\begin{equation}	\label{eq:pf_ip_grad2}
	\begin{split}
		&- \Brac{\nabla_j u (\bx)}\matT \nabla_j \phi (\bx_i) + \Brac{\nabla_j u (\bx)}\matT \mcL(\nabla_j u) (\bx) \\
		\geq~& \lambda M_{jj} \normo{ \nabla_j u(\bx) }_2^2 - \sum_{k\neq j} ( \lambda M_{jk} + 8 m \delta \varepsilon^{-1/2} {\bf{1}}_{j\sim k} ) \normo{ \nabla_j u(\bx) }_2 \normo{ \nabla_k u(\bx) }_2. 
	\end{split}
\end{equation}
Consider $\bx$ where $\normo{ \nabla_j u(\bx) }_2$ reaches its maximum, i.e. $ \normo{ \nabla_j u(\bx) }_2 = \norm{ \nabla_j u }_{L^\infty} $. The first order optimality condition reads
\[
0 = \nabla \Brac{ \normo{ \nabla_j u(\bx) }_2^2 } = \nabla \nabla_j u(\bx) \cdot \nabla_j u(\bx), 
\]
and the second order optimality condition reads 
\[
    0 \geq \Delta \Brac{ \normo{ \nabla_j u(\bx) }_2^2 } = 2 \normo{ \nabla \nabla_j u(\bx) }_{\rm F}^2 + 2 (\nabla_j u(\bx))\matT \Delta \nabla_j u(\bx).
\]
Thus, $ (\nabla_j u(\bx))\matT \Delta \nabla_j u(\bx) \leq 0$, since $\normo{ \nabla \nabla_j u(\bx) }_{\rm F}^2 \geq 0$. Under these conditions,  
\[
(\nabla_j u(\bx))\matT \mcL (\nabla_j u(\bx)) = (\nabla_j u(\bx))\matT \Delta \nabla_j u(\bx) + \nabla \log \pi_\varepsilon(\bx) \cdot \nabla \nabla_j u(\bx) \cdot \nabla_j u(\bx) \leq 0.
\]
Hence, at the maximum point \cref{eq:pf_ip_grad2} reads 
\begin{align*}
	\lambda M_{jj} \normo{ \nabla_j u }_{L^\infty}^2 -& \sum_{k\neq j} ( \lambda M_{jk} + 8 m \delta \varepsilon^{-1/2} {\bf{1}}_{j\sim k} ) \norm{ \nabla_j u }_{L^\infty} \normo{ \nabla_k u(\bx) }_2 \\
	\leq~& - \Brac{\nabla_j u (\bx)}\matT \nabla_j \phi (\bx_i) + \Brac{\nabla_j u (\bx)}\matT \mcL(\nabla_j u) (\bx) \\
	\leq~& \normo{\nabla_j u(\bx)}_2 \normo{\nabla_j \phi(\bx_i)}_2 + 0 \leq \delta_{ij} \normo{\nabla_j u}_{L^\infty} .
\end{align*}
Here, we use $ \normo{\nabla_j \phi(\bx_i)}_2 \leq \delta_{ij}$, since $\phi$ is $1$-Lipschitz and is only a function of $\bx_i$. So if $\normo{ \nabla_j u }_{L^\infty} > 0$, it holds that
\[
	\delta_{ij} \geq \lambda M_{jj} \normo{ \nabla_j u }_{L^\infty} - \sum_{k\neq j} ( \lambda M_{jk} + 8 m \delta \varepsilon^{-1/2} {\bf{1}}_{j\sim k} ) \normo{ \nabla_k u }_{L^\infty}. 
\]
Taking summation over $j \in \mcJ:= \{ j\in[b]: \normo{ \nabla_j u }_{L^\infty} > 0\}$ gives
\begin{equation}  \label{eq:tmp1}
	1 \geq \sum_{j\in \mcJ} \delta_{ij} \geq \sum_{j \in \mcJ} \normo{ \nabla_j u }_{L^\infty} \Rectbrac{ \lambda M_{jj} - \sum_{k\neq j} ( \lambda M_{kj} + 8 m \delta \varepsilon^{-1/2} {\bf{1}}_{j\sim k} ) }. 
\end{equation}
Notice $\# \{k: j \sim k\} \leq 4$, and we take $ \frac{\lambda}{\delta} \geq \frac{64m}{c\sqrt{\varepsilon}}$, so that 
\begin{equation}    \label{eq:diag_dominant}
\begin{split}
    &\lambda M_{jj} - \sum_{k\neq j} ( \lambda M_{kj} + 8 m \delta \varepsilon^{-1/2} {\bf{1}}_{j\sim k} ) \\
    \geq~& \lambda \Brac{ M_{jj} - \sum_{k\neq j} M_{kj} } - 32 m \delta \varepsilon^{-1/2} \geq \lambda c - 32 m \delta \varepsilon^{-1/2} \geq \frac{\lambda c}{2}.
\end{split}
\end{equation}
Together with \eqref{eq:tmp1}, we obtain \eqref{eq:PoisGrad}.
\end{proof}

\begin{lemma}	\label{lem:Inv_Ctrl}
There exists a dimension-independent constant $C_\pi$ s.t.\
\[
    \max_\alpha \mE_{\bx \sim \pi} \norm{\bD_\alpha \bx}_2^{-1} \leq C_\pi.
\]
\end{lemma}
\begin{proof}
Fix any $\alpha \in [d]$. For simplicity, denote
\[
\bD_\alpha \bx = (\bD_\alpha^{(v)} \bx , \bD_\alpha^{(h)} \bx ) = ( \bx_\alpha- \bx_{\alpha}^{(v)}, \bx_\alpha- \bx_{\alpha}^{(h)}).
\]  
Now introduce the change of variable $\bz = T \bx$ for some linear map determined via 
\[
\bz_{\alpha}^{(v)} = \bx_{\alpha}^{(v)} - \bx_\alpha, \quad \bz_{\alpha}^{(h)} = \bx_{\alpha}^{(h)} - \bx_\alpha,
\]
and $\bz_{\alpha-} = \bx_{\alpha-} \in \mR^{d-2}$ for the other coordinates. Accordingly, $\pi(\bx)$ is transformed into another distribution $\mu(\bz) = \pi(T^{-1}\bz)$ (note $\det T = 1$). Also note that $T^{-1}$ admits explicit form, i.e.
\[
\bx_{\alpha}^{(v)} = \bz_{\alpha}^{(v)} + \bz_\alpha, \quad \bx_{\alpha}^{(h)} = \bz_{\alpha}^{(h)} + \bz_\alpha, \quad \bx_{\alpha-} = \bz_{\alpha-}.
\]
Denote $\bz_{\alpha+} = (\bz_{\alpha}^{(v)} ,\bz_{\alpha}^{(h)} )$ for convenience. Consider the factorization
\[
\mu(\bz) = \mu(\bz_{\alpha+}, \bz_{\alpha-}) = \mu(\bz_{\alpha+} | \bz_{\alpha-}) \mu(\bz_{\alpha-}),
\]
where $\mu(\bz_{\alpha-})$ denotes the marginal of $\mu$ on $\bz_{\alpha-}$. Notice 
\[
\log \mu(\bz_{\alpha+} | \bz_{\alpha-}) = - \frac{\lambda}{2} \norm{ \by - \bA T^{-1} \bz }_2^2 - \delta \sum_{\alpha=1}^d \norm{ \bD_\alpha (T^{-1} \bz)}_2 - \log \mu (\bz_{\alpha-}) + \text{const}. 
\]
Fix $\bz_{\alpha -}$ for the moment. Notice $ \frac{\lambda}{2} \norm{ \by - \bA T^{-1} \bz }_2^2$ is $L$-smooth for some dimension-independent $L>0$, since only a dimension-independent number of coordinates of $\bA T^{-1} \bz$ depend on $\bz_{\alpha-}$. Therefore, fix any $\bz_{\alpha+}^0$, so that
\begin{align*}
	\frac{\lambda}{2} \norm{ \by - \bA T^{-1} (\bz_{\alpha +}, \bz_{\alpha -}) }_2^2 - \frac{\lambda}{2} \norm{ \by - \bA T^{-1} (\bz_{\alpha +}^0,\bz_{\alpha -})}_2^2 - \bv^0 \cdot (\bz_{\alpha +} - \bz_{\alpha +}^0) &\\
    \leq \frac{L}{2} \norm{ \bz_{\alpha +} - \bz_{\alpha +}^0}_2^2 &,
\end{align*}
where $\bv^0$ is the gradient w.r.t. $\bz_{\alpha +}$ of $\frac{\lambda}{2}\norm{ \by - \bA T^{-1} \bz }_2^2$ at $\bz_{\alpha +}^0$. Notice also 
\[
\delta \sum_{\alpha=1}^d \norm{ \bD_\alpha (T^{-1} (\bz_{\alpha +}, \bz_{\alpha -}) ) }_2 \leq \delta \sum_{\alpha=1}^d \norm{ \bD_\alpha ( T^{-1} (\bz_{\alpha +}^0, \bz_{\alpha -}) ) }_2 + 8 \delta \norm{ \bz_{\alpha +} - \bz_{\alpha +}^0}_2,
\]
since changing $\bz_{\alpha}^{(v)}$ or $\bz_{\alpha}^{(h)}$ affects $4$ finite difference terms in the summation. Combining the above controls, when $\bz_{\alpha+} \in B_1(\bz_{\alpha+}^0) := \{ \bz_{\alpha+} : \norm{ \bz_{\alpha+} - \bz_{\alpha+}^0 }_2 \leq 1 \}$, it holds that 
\begin{align*}
	\log \mu(\bz_{\alpha+} | \bz_{\alpha-}) -& \log \mu(\bz_{\alpha+}^0 | \bz_{\alpha-}) + \bv^0 \cdot (\bz_{\alpha +} - \bz_{\alpha +}^0) \\
    \geq& - \frac{L}{2} \norm{ \bz_{\alpha +} - \bz_{\alpha +}^0}_2^2 - 8 \delta \norm{ \bz_{\alpha +} - \bz_{\alpha +}^0}_2 \geq - \frac{L}{2} - 8 \delta. 
\end{align*}
Therefore, 
\begin{align*}
	1 =~& \int \mu(\bz_{\alpha+} | \bz_{\alpha-}) \mdd \bz_{\alpha+} \geq \int_{B_1(\bz_{\alpha+}^0)} \exp \Brac{ \log \mu(\bz_{\alpha+} | \bz_{\alpha-}) } \mdd \bz_{\alpha+} \\
	\geq~& \mu (\bz_{\alpha+}^0 | \bz_{\alpha-}) \exp \Brac{ - \frac{L}{2} - 8 \delta } \int_{B_1(\bz_{\alpha+}^0)} \exp \Brac{ - \bv^0 \cdot (\bz_{\alpha +} - \bz_{\alpha +}^0) } \mdd \bz_{\alpha+} \\
	\geq~& \mu (\bz_{\alpha+}^0 | \bz_{\alpha-}) \exp \Brac{ - \frac{L}{2} - 8 \delta } |B_1|, 
\end{align*}
where we use the Jensen's inequality and the symmetry of $B_1(\bz_{\alpha-}^0)$. Therefore, 
\[
\mu (\bz_{\alpha+}^0 | \bz_{\alpha-}) \leq |B_1|^{-1} \exp \Brac{ \frac{L}{2} + 8 \delta } .
\]
This holds for arbitrary $\bz_{\alpha+}^0$, so that the marginal distribution of $\mu(\bz_{\alpha+})$ satisfies 
\[
    \mu(\bz_{\alpha+}) = \int \mu(\bz_{\alpha+} | \bz_{\alpha-}) \mu(\bz_{\alpha-}) \mdd \bz_{\alpha-} \leq |B_1|^{-1} \exp \Brac{ \frac{L}{2} + 8 \delta } =: C'.
\]
Note $C'$ is dimension-independent. Finally, notice 
\begin{align*}
	\mE_{\bx \sim \pi} \norm{\bD_\alpha \bx}_2^{-1} =~& \int_{\mR^2} \frac{\mu(\bz_{\alpha+})}{\norm{ \bz_{\alpha+} }_2 } \mdd \bz_{\alpha}^{(v)} \mdd \bz_{\alpha}^{(h)}  \leq 1 + \int_{ \norm{ \bz_{\alpha+} }_2 \leq 1} \frac{C'}{\norm{ \bz_{\alpha+} }_2 } \mdd \bz_{\alpha}^{(v)} \mdd \bz_{\alpha}^{(h)} \\
	=~& 1 + \int_0^{2\pi} \int_0^1 \frac{C'}{r} \cdot r \mdd r \mdd \hat\alpha = 1 + 2 \pi C' =: C_\pi.
\end{align*}
Thus, $C_\pi$ is dimension-independent. This completes the proof. 
\end{proof}

\subsection{\texorpdfstring{Proof of \cref{lem:cond_one_pixel}}{Proof of Lemma 5.1}}\label{app:loc_cond}

In this proof, we work with the matrix representation $\bX\in\mathbb{R}^{n \times n}$ for the unknown image such that the pixel at the index tuple $(\alpha,\beta) \in [n] \times [n]$ is given by $\bX_{\alpha,\beta}$.
Moreover, we use the notation $(\alpha\pm r) \coloneqq (\alpha-r, \alpha-r+1,\dots, \alpha+r-1, \alpha+r)$,
and we let $w:(0\pm r)\times (0\pm r)\to\mathbb{R}_{>0}$ be the weights of the discrete PSF of the convolution kernel.

With this notation, the maximal cliques in the posterior factorization \cref{eq:clique_pots} can be written as
\begin{align}\label{eq:max_clique_mat}
\begin{split}
	V_{\alpha,\beta}( \bX ) &= \frac\lambda2 ( \bY_{\alpha,\beta} - \sum_{\mu,\nu=-r}^{r} w(\mu,\nu) \bX_{\alpha +\mu, \beta + \nu} )^2 \\
	&~~~~+ \delta \sqrt{ (\bX_{\alpha+1,\beta}-\bX_{\alpha,\beta})^2 + (\bX_{\alpha,\beta+1}-\bX_{\alpha,\beta})^2 + \varepsilon}.
\end{split}
\end{align}

We now consider the full conditional of a fixed pixel 
$\pi(\bX_{\hat\alpha,\hat\beta}|\bX_{\setminus(\hat\alpha,\hat\beta)})$, where 
$\setminus(\hat\alpha,\hat\beta) \coloneqq [n] \times [n] \setminus (\hat\alpha,\hat\beta)$.
To find an expression for $\pi(\bX_{\hat\alpha,\hat\beta}|\bX_{\setminus(\hat\alpha,\hat\beta)})$, it suffices to find all maximal cliques $V_{\alpha,\beta}$ which depend on $\bX_{\hat\alpha,\hat\beta}$.
This follows from the Hammersley-Clifford theorem \cite{hammersley1971markov} and the fact that the posterior \cref{eq:smooth_pos_expl} is a Gibbs density.
Therefore, let $\Theta_{\hat\alpha,\hat\beta} \subseteq [n] \times [n]$ be the set of all maximal cliques $V_{\alpha,\beta}$ which depend on $\bX_{\hat\alpha,\hat\beta}$, and let $\Phi_{\hat\alpha,\hat\beta} \subseteq [n] \times [n]$ be the set of pixels on which the cliques $\{V_{\alpha,\beta}\,|\,(\alpha,\beta)\in \Theta_{\hat\alpha,\hat\beta}\}$ depend, i.e., $\Phi_{\hat\alpha,\hat\beta}$ is the neighborhood of $\bX_{\hat\alpha,\hat\beta}$.
Thus, we can write 
$$
\pi(\bX_{\hat\alpha,\hat\beta}|\bX_{\setminus(\hat\alpha,\hat\beta)}) 
\propto \pi(\bX_{\hat\alpha,\hat\beta}|\bX_{\Phi_{\hat\alpha,\hat\beta}\setminus(\hat\alpha,\hat\beta)}) 
\propto \exp{( - \sum_{(\alpha,\beta)\in\Theta_{\hat\alpha,\hat\beta}} V_{\alpha,\beta}( \bX_{\Phi_{\hat\alpha,\hat\beta}} ) )},
$$
for some $\Theta_{\hat\alpha,\hat\beta}$ and $\Phi_{\hat\alpha,\hat\beta}$, which we now specify.

In fact, $\Theta_{\hat\alpha,\hat\beta}$ and $\Phi_{\hat\alpha,\hat\beta}$ are determined by the first term in \cref{eq:max_clique_mat}.
The clique $V_{\alpha,\beta}$ depends on $\bX_{\hat\alpha,\hat\beta}$ if $\alpha = \hat\alpha - \mu$ and $\beta = \hat\beta - \nu$ where $\mu,\nu\in[-r,-r+1, \dots, r]$.
Therefore, $\Theta_{\hat\alpha,\hat\beta}=\hat\alpha\pm r \times \hat\beta\pm r$.
From this follows, by checking the sum in the first term of \cref{eq:max_clique_mat}, that $\Phi_{\hat\alpha,\hat\beta}=\hat\alpha\pm 2r \times \hat\beta\pm 2r$.

\end{document}